%% file: note-reducible_2025_10_16.tex
\documentclass[12pt]{article}

\usepackage[T1]{fontenc}
\usepackage[utf8]{inputenc}

\usepackage{fourier} 
\usepackage{soul} 
\usepackage{enumerate} 

\usepackage{amsfonts,amssymb,amsmath,amsthm}
\usepackage{bbm}
\usepackage{pgf}

\usepackage{hyperref}
\usepackage{cancel}

\usepackage{todonotes}

\usepackage{xcolor}
\definecolor{darkblue}{rgb}{0,0,0.5}
\usepackage{transparent}
\usepackage{graphicx}
\newcommand{\executeiffilenewer}[3]{%
  \ifnum\pdfstrcmp{\pdffilemoddate{#1}}%
  {\pdffilemoddate{#2}}>0%
  {\immediate\write18{#3}}\fi%
}

\numberwithin{equation}{section}
\def\PP{\mathbb{P}}

\def\RR{\mathbb{R}}
\def\NN{\mathbb{N}}
\def\ZZ{\mathbb{Z}}

\def\EE{\mathbb{E}}
\def\11{\mathbbm{1}}

\def\E{\mathbb{E}}
\def\P{\mathbb{P}}

\def\N{\mathbb{N}}

\def\d{\partial}
\def\Z{\mathbb{Z}}

\newcommand{\aj}{j}

\newtheorem{thm}{Theorem}[section]
\newtheorem{lem}[thm]{Lemma}
\newtheorem{cor}[thm]{Corollary}

\newtheorem{prop}[thm]{Proposition}

\theoremstyle{remark}
\newtheorem{rem}{Remark}

\newcommand{\tJ}{\mathfrak{J}}
\newcommand{\tP}{\mathfrak{P}}
\newcommand{\tQ}{\mathfrak{Q}}
\newcommand{\tR}{\mathfrak{R}}
\newcommand{\tS}{\mathfrak{S}}

\newcommand{\dd}{\mathrm{d}}

\newcommand{\vertiii}[1]{{\left\vert\kern-0.25ex\left\vert\kern-0.25ex\left\vert #1 
    \right\vert\kern-0.25ex\right\vert\kern-0.25ex\right\vert}}

    \def\restriction#1#2{\mathchoice
                  {\setbox1\hbox{${\displaystyle #1}_{\scriptstyle #2}$}
                  \restrictionaux{#1}{#2}}
                  {\setbox1\hbox{${\textstyle #1}_{\scriptstyle #2}$}
                  \restrictionaux{#1}{#2}}
                  {\setbox1\hbox{${\scriptstyle #1}_{\scriptscriptstyle #2}$}
                  \restrictionaux{#1}{#2}}
                  {\setbox1\hbox{${\scriptscriptstyle #1}_{\scriptscriptstyle #2}$}
                  \restrictionaux{#1}{#2}}}
    \def\restrictionaux#1#2{{#1\,\smash{\vrule height .8\ht1 depth .85\dp1}}_{\,#2}}

\newcommand\xqed[1]{%
    \leavevmode\unskip\penalty9999 \hbox{}\nobreak\hfill
    \quad\hbox{#1}}
\newcommand\erem{\xqed{$\triangle$}}

\begin{document}

\title{Quasi-stationary distributions in reducible state spaces}

\author{Nicolas Champagnat$^{1}$, Denis Villemonais$^{1,2,3}$}

\footnotetext[1]{Universit\'e de Lorraine, CNRS, Inria, IECL, F-54000 Nancy, France}
\footnotetext[2]{Université de Strasbourg, IRMA, Strasbourg, France}  
\footnotetext[3]{Institut Universitaire de France\\
	E-mail: Nicolas.Champagnat@inria.fr, Denis.Villemonais@unistra.fr}

\maketitle

\begin{abstract}
  We study quasi-stationary distributions and quasi-limiting behavior of Markov chains in general reducible state spaces with
  absorption. 
  
  Firstly, we consider state spaces that can be decomposed into two successive
  subsets (between which communication is only possible in a single direction), differentiating between three situations: either the process exits the first set at a higher pace, or the second set at a higher pace, or both space at a comparable pace. These first results allow us to characterize the
  exponential order of magnitude and the exact polynomial correction, called polynomial convergence parameter, for the leading order
  term of the semigroup for large time. We also provide explicit convergence speeds to this leading order term. 
  
  Secondly, we consider  general Markov chains with finitely or countably many communication classes by applying the first results iteratively over the
  communication classes of the chain. We  characterize explicitely the polynomial convergence parameter, determine the
  complete set of quasi-stationary distributions and provide explicit estimates for the speed of convergence to quasi-limiting
  distributions in the case of finitely many communication classes.
  
  We conclude with an application of these results to the case of
  denumerable state spaces, where we prove that, in general, there is existence of a quasi-stationary distribution
  without assuming irreducibility before absorption. This  holds true assuming only aperiodicity, the existence of a Lyapunov
  function and the existence of a point in the state space from which the return time is finite with positive probability.
\end{abstract}




\section{Introduction}
\label{sec:intro}

Let $(X_n,n\in \ZZ_+)$ be a Markov chain in $D\cup\{\d\}$ where $D$ is a measurable space, $\d\not\in D$ and $\ZZ_+:=\{0,1,\ldots\}$. For all $x\in D\cup\{\d\}$, we denote as usual by $\PP_x$ the law of $X$ given $X_0=x$ and for any probability measure $\mu$ on $D\cup\{\d\}$, we define $\PP_\mu=\int_{D\cup\{\d\}}\PP_x\,\mu(dx)$. We also denote by $\E_x$ and $\E_\mu$ the associated expectations. We assume that $\d$ is absorbing, which means that $X_n=\d$ for all $n\geq \tau_\d$, $\PP_x$-almost surely, where 
\[ 
\tau_\d=\inf\{n\in \mathbb Z_+,\,X_n=\d\}.  
\]
We study the sub-Markovian transition semigroup of $X$, $(S_n)_{n\in\Z_+}$, defined as
\begin{align*}
S_n f(x)=\E_x\left(f(X_n)\11_{n<\tau_\d}\right),\ \forall n\in\Z_+,
\end{align*}
for all bounded or nonnegative measurable function $f$ on $D$ and all $x\in D$. We also define as usual the left-action of $P_n$ on
measures as
\[
\mu S_n f=\EE_\mu\left(f(X_n)\11_{n<\tau_\d}\right)=\int_D S_nf(x)\,\mu(dx),
\]
for all probability measure $\mu$ on $D$ and all bounded or nonnegative measurable function~$f:D\to \RR$.

The purpose of this article is to provide original and practical criteria allowing to study the quasi-limiting behaviour of absorbed, reducible Markov processes in general state spaces, both in cases of geometric and polynomial convergence in total variation to a
quasi-stationary distribution.

We recall that a quasi-stationary distribution (QSD) for $X$ is a probability measure $\nu_{QS}$ on $D$ such that
\[
\PP_{\nu_{QS}}(X_n\in\cdot\mid n<\tau_\d)=\frac{\nu_{QS} S_n}{\nu_{QS} S_n\11_D}=\nu_{QS},\ \forall n\geq 0.
\]
It is well known that a probability measure $\nu_{QS}$ is a QSD for $X$ if and only if it is a  quasi-limiting distribution (see
e.g.~\cite{ColletMartinezEtAl2013,MeleardVillemonais2012}). By a  quasi-limiting distribution $\nu$, we mean a probability measure $\nu$ such that, for some probability measure $\mu$ on $D$ and for any measurable subset $\Gamma\subset D$, the
conditional probabilities $\PP_\mu(X_n\in \Gamma\mid n<\tau_\d)$ converges to $\nu(\Gamma)$.
To each QSD $\nu_{QS}$ is associated an \emph{exponential convergence parameter}  $\theta\in(0,1]$, such that 
\[
\mathbb{P}_{\nu_{QS}}(\tau_\d\geq n)=\theta^n,\ \forall n\geq 0.
\]
This parameter is called a \textit{convergence parameter} in~\cite{NiemiNummelin1986}, and we add the term \textit{exponential} to distinguish it from the \textit{polynomial convergence parameter} that we introduce below.

The study of quasi-limiting behaviour of Markov chains on reducible state spaces started with the work of
  Mandl~\cite{Mandl1959} (see also~\cite{darroch-seneta-65}). Since then, several works studied cases of finite state
  spaces~\cite{Ogura1975,CattiauxMeleard2010,ChampagnatDiaconisEtAl2012,BenaimCloez2015,BenaimCloezEtAl2016,DoornPollett2008,DoornPollett2009}
  or infinite state spaces~\cite{Gosselin2001,ChampagnatRoelly2008}. Most of these works are devoted to spectific processes, while the
  articles~\cite{DoornPollett2008,DoornPollett2009} address the general situation in finite state spaces (see also the
  survey~\cite{DoornPollett2013}).

 In order to obtain results on general state spaces, we make use of results on the principal
eigenvalue and eigenvectors of iterates of upper triangular matrices of linear operators over a Banach space. This allows us to prove sufficient conditions ensuring that a reducible process $X$ satisfies
\begin{equation}
\label{eq:cv-subgeom}
\left\|\theta^{-n} n^{-j(x)} \PP_x(X_n\in\cdot)-\eta(x)\nu_{QS}\right\|\xrightarrow[n\to+\infty]{} 0,\quad \forall x\in D,
\end{equation}
for some QSD $\nu_{QS}$ (which may depend on $x$), some measurable functions $\eta:D\to[0,+\infty)$ and $j:D\to\ZZ_+=\{0,1,\ldots\}$,
and where $\|\cdot\|$ is a weighted total variation norm (see Assumption~(A) in Section~\ref{sec:expopopar} for more details). We
call the function $j$ the \textit{polynomial convergence parameter}. and prove several properties of $j$, $\eta$ and $\nu_{QS}$ in
Section~\ref{sec:expopopar}.

We emphasize that for many usual irreducible Markov processes, the quasi-limiting behaviour is well understood and it is known that
this result holds true with $j\equiv 0$ and a QSD $\nu_{QS}$ independent of $x$ (see for
instance~\cite{ColletMartinezEtAl2013,MeleardVillemonais2012,ChampagnatVillemonais2016b,ChampagnatVillemonais2017b}). This is also
true for some reducible processes with exponential convergence (see~\cite[Thm\,6.1]{ChampagnatVillemonais2017b}). { \color{black}
  Compared to the existing results of the literature, our goal is to provide complete results applying to general processes, as done
  in finite state spaces in~\cite{DoornPollett2008,DoornPollett2009,DoornPollett2013}. Compared to these works, we consider reducible
  processes in general state spaces that can be decomposed into finitely or denumerably many communication classes. We are also able
  to characterize explicitely the polynomial convergence parameter $j$ and the possibly subgeometric convergence rate associated to
  each communication class and we fully characterize the support of $\eta$ and more generally the sets of initial conditions where
  the survival probability has some given asymptotic behavior. Our results also apply to processes with denumerably many
  communication classes provided that only finitely many of them have maximal exponential convergence parameter. We make a more
  detailed review of the results of the literature in the beginning of Section~\ref{sec:finite-classes} and we elaborate on the
  novelties of our work compared to~\cite{DoornPollett2008,DoornPollett2009,DoornPollett2013} after stating our main result,
  Theorem~\ref{thm:finite-classes}, in Section~\ref{sec:finite-classes}. }
{\color{black}Finally, we emphasize that,} following the same approach as in~\cite{ChampagnatVillemonais2019},
all the results of the present paper can be easily extended to non sub-Markov semigroups.

The paper is organized as follows. In Section~\ref{sec:expopopar}, we present our main assumption and its first consequences. In
Section~\ref{sec:three-sets}, we consider reducible sub-Markov processes with two successive sets where this assumption is satisfied.
{\color{black}Three cases are considered in Subsections~\ref{sec:D1source} to~\ref{sec:D1criticalsink} depending on how the exponential
convergence parameter of the two successive sets compare.}
We then consider in Section~\ref{sec:finite-classes} reducible sub-Markov processes with several communication classes. As an
application, we prove in Section~\ref{sec:discretelyap} that, under a mild Lyapunov assumption, processes on discrete state spaces
always admit quasi-limiting distributions.


\bigskip\noindent\textbf{Notation.} The set $\mathcal M(D)$ is the Banach space of finite signed measures over $D$, endowed with the
total variation norm. We denote by $\mathcal M_+(D)\subset \mathcal M(D)$ the set of non-negative finite measures over $D$.  Given a
positive measurable function $W$, the set $\mathcal M(W)$ is the Banach space of signed measures $\mu$ such that $|\mu|(W)<+\infty$, endowed with the norm
\[
\|\mu\|_W:=|\mu|(W).
\]
We extend the operator $S_n$ to $\mathcal M(D)$ by $\mu S_n=\int_D \delta_x S_n \,\mu(\mathrm dx)$.
The set $L^\infty(W)$ is the Banach space of measurable functions $f$ such that $\|f/W\|_\infty<+\infty$, endowed with norm
\[
\|f\|_W:=\|f/W\|_\infty.
\]
Because of the nature of our problem, we will often consider the extensions to $D\cup\{\d\}$ of functions defined on a subset of $D\cup\{\d\}$. Systematically and without further notice, all functions are extended by the value $0$ outside of their domain of definition. In all the sequel, $C$ will denote a finite constant that may change from line to line.

\section{Exponential and polynomial convergence parameter}
\label{sec:expopopar}


The \textit{exponential convergence parameter} of the semigroup $(S_n)_{n\in\N}$  {\color{black} is given as a function of $\mu\in {\cal M}_+(D)$} by
\[
\theta_S(\mu):=\inf\left\{\theta\geq 0,\  \liminf_{n\to+\infty} \theta^{-n} \mu S_n\11_D=0\right\}.
\]
We also set $\theta_{0,S}=\sup_{x\in D} \theta_{S}(x)$, where $\theta_S(x)=\theta_S(\delta_x)$. We define the \textit{polynomial
  convergence parameter} of the semigroup $(S_n)_{n\in\N}$ {\color{black} as a function of $\mu\in {\cal M}_+(D)$} by
\begin{equation}
  \label{eq:def-j_S}
  j_S(\mu):=\inf\{\ell\geq 0,\  \liminf_{n\to+\infty} n^{-\ell} \theta_{0,S}^{-n}\mu S_n\11_D=0\}.
\end{equation}
with the convention $\inf\emptyset=+\infty$.
We also set $j_{0,S}=\sup_{x} j_S(x)$, where $j_S(x):=j_S(\delta_x)$. Note that if $\theta_S(\mu)<\theta_{0,S}$, then
 $j_S(\mu)=0$. We will see in Proposition~\ref{prop:propajs1} below that the converse inequality $\theta_S(\mu)>\theta_{0,S}$ never happens.

In this section, we are interested in the implications of the following assumption~(A) on $j_S$ and on the existence and convergence toward a quasi-statio\-nary distribution for $X$. In the following sections, we will study sufficient properties implying that $X$ satisfies this condition.

\medskip \noindent \textbf{Assumption (A).} We have $\theta_{0,S}\in(0,1]$, $j_S$ is integer valued and there exist a measurable
function $W_S:D\rightarrow[1,+\infty)$, 
a finite or countable set $I_S$ and some probability measures $\nu_{S,i}\in\mathcal
M(W_S)$ and non-identically zero non-negative $\eta_{S,i}\in L^\infty(W_S)$ for each $i\in I_S$, such that
\begin{align}
\label{eq:etanu}
\sum_{i\in I_S} \eta_{S,i}\nu_{S,i}(W_S)\in L^\infty(W_S)
\end{align}
and such that, for all $f\in L^\infty(W_S)$, all $n\geq 1$ and all $x\in
D$,
\begin{equation}
\label{eq:eta}
\left|\theta_{0,S}^{-n} n^{-j_S(x)} \EE_x(f(X_n)\11_{n< \tau_\d})-\sum_{i\in I_S} \eta_{S,i}(x)\nu_{S,i}(f)\right|\leq \alpha_{S,n} W_S(x)\|f\|_{W_S},
\end{equation}
where $\alpha_{S,n}$ goes to $0$ when $n\to+\infty$.

\medskip When Assumption~(A) holds true, we define
\begin{align}
\label{eq:defeatS}
\eta_S:=\sum_{i\in I_S} \eta_{S,i}\in L^\infty(W_S),
\end{align}
{\color{black} where $\eta_S\in L^\infty(W_S)$ is a consequence of~\eqref{eq:eta} with $f\equiv 1$.}
Note that~\eqref{eq:eta} only gives an equivalent of $\delta_xS_n\11_D$ when $\eta_{S,i}(x)>0$ for at least one
$i\in I_S$, or equivalently when $\eta_S(x)>0$. In particular, for all $x$ such that $\theta_S(x)<\theta_{S,0}$,~\eqref{eq:eta} implies that $\eta_{S,i}(x)=0$ for all
$i\in I_S$.

We also emphasize that, for all $x\in D$ such that $\eta_S(x)>0$,~\eqref{eq:eta} entails that $\PP_x(X_n\in\cdot
  \mid n< \tau_\d)=\frac{\delta_x S_n}{\delta_x S_n\mathbbm{1}_D}$ converges in $\mathcal M(W_S)$ toward $\frac{1}{\eta_S(x)}\sum_{i\in I_S} \eta_{S,i}(x)\nu_{S,i}$, which is thus a quasi-limiting distribution and hence a
quasi-stationary distribution. The rest of this section is dedicated to the exposition and proofs of finer properties on $j_S$ and on
the quasi-stationary distributions of $X$ under Assumption~(A).

\medskip

\begin{rem}
    The results of this paper are stated in the discrete-time setting. The adaptation to the continuous time setting can be obtained  by considering Assumption~(A) for the included Markov chain and by assuming in addition that, for all $t\in [0,1]$, $\E_x(W_S(X_t))\leq C W_S(x)$ for some constant $C>0$.
    \erem
\end{rem}

\medskip

We start our study with simple properties on the polynomial convergence parameter $j_S$. 

\begin{prop}
    \label{prop:propajs1}
    For all $\mu\in\mathcal M_+(D)$,
    \begin{align}
    \label{eq:prop11theta}
    \theta_S(\mu)\geq \sup\Big\{\theta\geq 0,\ \mu\{x,\ \theta_S(x)\geq \theta\}>0\Big\}
    \end{align}
    and
    \begin{align}
    \label{eq:prop11}
    j_S(\mu)\geq \sup\Big\{\ell\geq 0,\ \mu\{x,\ \aj_S(x)\geq \ell\}>0\Big\}
    \end{align}
    If Assumption~(A) holds true, then $j_S$ is lower semi-continuous on $\mathcal M_+(W_S)$ and, for all $\mu\in \mathcal M_+(W_S)$,
     \begin{align}
    \label{eq:prop12theta}
    \theta_S(\mu)= \sup\Big\{\theta\geq 0,\ \mu\{x,\ \theta_S(x)\geq \theta\}>0\Big\}
    \end{align}
    and
    \begin{align}
    \label{eq:prop12}
    \aj_S(\mu)=\sup\Big\{\ell\geq 0,\ \mu\{x,\ \aj_S(x)\geq \ell\}>0\Big\}.
    \end{align}
    In addition,
    \begin{align}
    \label{eq:prop13}
    \aj_S(\mu)=\aj_S(\mu S_1)
    \end{align}
    and $(\aj_S(X_n))_{n\geq 0}$ is $\P_x$-almost surely non-increasing, for all $x\in D$.
\end{prop}


\begin{proof}[Proof of Proposition~\ref{prop:propajs1}]
    We prove~\eqref{eq:prop11}, \eqref{eq:prop12} and~\eqref{eq:prop13} in this order. The proof of~\eqref{eq:prop11theta}, \eqref{eq:prop12theta} are similar and thus omitted.
    
    \medskip\noindent\textit{Proof of~\eqref{eq:prop11}.}
    Fix a positive measure $\mu$ on $D$ (the result is trivial if $\mu=0$). 
    For all $\varepsilon>0$ and for all $x\in D$ such that $\aj_S(\mu)+\varepsilon<\aj_S(x)$, we have by definition of $\aj_S(x)$ and the fact that $(\aj_S(\mu)+\varepsilon+\aj_S(x))/2<\aj_S(x)$,
    \begin{align*}
    \liminf_{n\to+\infty}  \theta_{0,S}^{-n}n^{-(\aj_S(\mu)+\varepsilon+\aj_S(x))/2}\delta_x S_n\11_D>0
    \end{align*}
    and hence, since $(j_S(\mu)+\varepsilon+j_S(x))/2>j_S(\mu)+\varepsilon$,
    \begin{align*}
    \liminf_{n\to+\infty}  \theta_{0,S}^{-n}n^{-\aj_S(\mu)-\varepsilon}\delta_x S_n\11_D=+\infty.
    \end{align*}
    Using Fatou's Lemma, we obtain
    \begin{align*}
    0=\liminf_{n\to+\infty} \  \theta_{0,S}^{-n}n^{-\aj_S(\mu)-\varepsilon}\mu S_n\11_D
    &\geq \mu\Big(\liminf_{n\to+\infty} \  \theta_{0,S}^{-n}n^{-\aj_S(\mu)-\varepsilon} S_n\11_D\Big)\\
    &\geq \mu\big(+\infty\, \11_{\aj_S(\cdot)>\aj_S(\mu)+\varepsilon}\big).
    \end{align*}
    This implies that, for all $\varepsilon>0$, $\mu\{x,\ \aj_S(x)>\aj_S(\mu)+\varepsilon\}=0$, and hence that $\mu\{x,\ \aj_S(x)>\aj_S(\mu)\}=0$.
    In particular, any $\ell\geq 0$ such that $\mu\{x,\aj_S(x)\geq \ell\}>0$ satisfies $\ell\leq \aj_S(\mu)$. We thus proved that
    \begin{align}
    \label{eq:ajsbigger}
    \aj_S(\mu)\geq \ell_\mu:=\sup\Big\{\ell\geq 0,\ \mu\{x,\ \aj_S(x)\geq \ell\}>0\Big\}.
    \end{align}

    \medskip\noindent\textit{Proof of~\eqref{eq:prop12} and the fact that $j_S$ is lower semi-continuous.} We assume that $\mu\in
    \mathcal M_+(W_S)$ is a positive measure and that Assumption~(A) holds true, and we prove $\aj_S(\mu)\leq \ell_\mu$, where
    $\ell_\mu$ is defined in~\eqref{eq:ajsbigger}. Fix
    $\ell>\ell_\mu$, so 
    $\aj_S(x)<\ell$ $\mu(\mathrm dx)$-almost everywhere.     
    Then, by Assumption~(A),
    \begin{align*}
    \left|\theta_{0,S}^{-n}n^{-\ell}\delta_xS_n\11_D\right|
    &\leq n^{-(\ell-\aj_S(x))}\left(\alpha_{S,n}W_S(x)+\eta_S(x)\right)\\
    &\leq n^{-(\ell-\aj_S(x))} C\,W_S(x)\\
    &\xrightarrow[n\to+\infty]{\text{$\mu(\mathrm dx)-$a.e.}} 0
    \end{align*}
    for some constant $C>0$. This also implies that $\left|\theta_{0,S}^{-n}n^{-\ell}\delta_xS_n\11_D\right|$ is bounded, up to a multiplicative constant, by the $\mu$-integrable function $W_S$, 
    and hence, by the Lebesgue dominated convergence theorem, we obtain
    \begin{align*}
    \lim_{n\to+\infty} \theta_{0,S}^{-n}n^{-\ell}\mu S_n\11_D=0.
    \end{align*}
    This entails that $\ell\geq \aj_S(\mu)$. Since $\ell>\ell_\mu$ was arbitrary, we deduce that $\ell_\mu\geq \aj_S(\mu)$. This concludes the proof of~\eqref{eq:prop12}.
    
    Now let $(\mu_n)_{n\in\NN}$ be a sequence of elements of $\mathcal M_+(W_S)$ converging toward $\mu$ in $\mathcal M_+(W_S)$. Then for all measurable subset $A\subset D$, we have $\mu_n(A)\to \mu(A)$ and hence, for all $\ell\geq 0$ such that $\mu\{x,\ j_S(x)\geq \ell\}>0$, 
    \[
    \liminf_{n\geq +\infty} \mu_n\{x,\ j_S(x)\geq \ell\}>0,
    \]
    so that
    \[
    \liminf_{n\geq +\infty} j_S(\mu_n)\geq \ell.
    \]
    This holds true for all $\ell<j_S(\mu)$, so
    \[
    \liminf_{n\geq +\infty} j_S(\mu_n)\geq j_S(\mu),
    \]
    which concludes the proof of the fact that $j_S$ is lower semi-continuous.

    \medskip\noindent\textit{Proof of~\eqref{eq:prop13} and that $j_S(X_n)$ is a.s.\ non-increasing.} We still assume that $\mu\in  \mathcal M_+(W_S)$ and that Assumption~(A) holds true. Let us first prove that $\aj_S(\mu)=\aj_S(\mu S_1)$. We have, for all $\ell\geq 0$,
    \begin{align*}
    \theta_{0,S}^{-n} n^{-\ell}(\mu S_1)S_n&=\left(\frac{n}{n+1}\right)^\ell\theta_{0,S}\,\theta_{0,S}^{-(n+1)}(n+1)^{-\ell}\mu S_{n+1}\\
    &\sim_{n\to+\infty} \theta_{0,S}\,\theta_{0,S}^{-(n+1)}(n+1)^{-\ell}\mu S_{n+1}.
    \end{align*}
    This implies that the $\liminf$ of $\theta_0^{-n} n^{-\ell}(\mu S_1)S_n$ equals $0$ if and only if the $\liminf$ of $\theta_{0,S}^{-n}\,n^{-\ell}\mu S_{n}$ equals $0$, and hence
    that $\aj_S(\mu S_1)=\aj_S(\mu)$.
    
    We conclude by proving the last assertion of the proposition. We have
    \[
    \aj_S(\mu S_1)=\sup\Big\{\ell\geq 0,\ \mu S_1\{x,\ \aj_S(x)\geq \ell\}>0\Big\},
    \]
    hence
    \[
    \mu S_1\{x,\ \aj_S(x)>\aj_S(\mu S_1)\}=0.
    \]
    Using the equality  $\aj_S(\mu S_1)=\aj_S(\mu)$ and the fact that $\mu S_1=\P_\mu(X_1\in\cdot,\ X_1\neq\d)$, we deduce that
    \[
    \P_\mu(\aj_S(X_1)>\aj_S(\mu))=0,
    \]
    where we used $j_S(\d)=0$ due to our notational convention about extension of functions.
    This and a straightforward application of the Markov property concludes the proof of the proposition.
\end{proof}

We now turn our attention to the implications of Assumption~(A) for quasi-stationary distributions. 
The following proposition considers quasi-stationary distributions $\nu\in \mathcal M_+(W_S)$ such that $\nu(\eta_S)>0$.

\begin{prop}
    \label{prop:QSDbisrec}
    Assume that Assumption~(A) holds true and let $\nu\in \mathcal M_+(W_S)$ be a quasi-stationary distribution  such that
    $\nu(\eta_S)>0$ and such that $\nu\{j_S(\cdot)\leq \ell\}=1$ for some $\ell\geq 0$. Then the exponential absorption parameter of $\nu$ is $\theta_{0,S}$.
\end{prop}

\begin{proof}
    According to~\eqref{eq:eta}, we have for all $x\in D$  
    \begin{align}
    \label{eq:equation1}
    \theta_{0,S}^{-n} n^{-\aj_S(x)} \EE_x(\11_D(X_n)\11_{n< \tau_\d})\xrightarrow[n\to+\infty]{} \eta_S(x),
    \end{align}
    and
    \begin{align}
    \label{eq:equation2}
    \left|\theta_{0,S}^{-n} n^{-\aj_S(x)} \EE_x(\11_D(X_n)\11_{n< \tau_\d})\right|&\leq |\eta_S(x)|+\alpha_{S,n} |W_S(x)|\, \|\11_D\|_{W_S}\leq C W_S(x).
    \end{align}
    Denote by $\theta_\nu$ the absorption parameter of $\nu$. Assume that $\theta_\nu>\theta_{0,S}$, then, for any $\ell\geq 1$ such that $\nu\{\aj_S(\cdot)\leq \ell\}=1$, and $n\geq 1$ large enough so that $\theta_\nu^{-n}\leq \theta_{0,S}^{-n}n^{-\ell-1}$,
    \begin{align*}
    \nu(D)
    &=\theta_{\nu}^{-n} \EE_\nu(\11_D(X_n)\11_{n< \tau_\d})\\
   &\leq n^{-1}\nu\left(\theta_{0,S}^{-n} n^{-\aj_S(\cdot)} \EE_\cdot(\11_D(X_n)\11_{n< \tau_\d})\right)\xrightarrow[n\to+\infty]{} 0,
    \end{align*}
    using~\eqref{eq:equation2}. This is a contradiction and hence $\theta_\nu\leq \theta_{0,S}$.
    
    Assume now that $\theta_\nu<\theta_{0,S}$, then  Fatou's lemma entails that
    \begin{align*}
    1&=\nu(\11_D)=
    \theta_{\nu}^{-n} \EE_\nu(\11_D(X_n)\11_{n< \tau_\d})\\
    &\geq\nu\left(\theta_{\nu}^{-n} n^{-\aj_S(\cdot)} \EE_\cdot(\11_D(X_n)\11_{n< \tau_\d})\right)\xrightarrow[n\to+\infty]{}   +\infty
    \end{align*}
    by~\eqref{eq:equation1} and since $\nu(\eta_S)>0$. Hence, we have proved that $\theta_\nu=\theta_{0,S}$.
\end{proof}

 The following proposition shows  that all quasi-stationary distributions in $\mathcal M_+(W_S)$  with parameter~$\theta_{0,S}$ are convex combinations of the $\nu_{S,i}$.

\begin{prop}
    \label{prop:QSDbis} Assume that there exists a QSD $\nu$ with exponential absorption parameter $\theta_{0,S}$. Then
    \[
    \aj_S(\nu)=0\quad\text{and}\quad\nu\left\{\aj_S(\cdot)>0\right\}=0.
    \]
    If in addition Assumption~(A) holds true and $\nu\in \mathcal M_+(W_S)$, then $\nu(\eta_S)=1$ and $\nu=\sum_{i\in I_S}\nu(\eta_{S,i})\nu_{S,i}$. 
\end{prop}

\begin{proof}
    The property $\aj_S(\nu)=0$ is immediate, while the second equality derives immediately from~\eqref{eq:prop11} in Proposition~\ref{prop:propajs1}.
    
    If in addition Assumption~(A) holds true and $\nu\in \mathcal M_+(W_S)$, then~\eqref{eq:equation1} and~\eqref{eq:equation2}
    are satisfied.
    Using the fact that $\nu(W_S)<+\infty$, we deduce from Lebesgue's dominated convergence theorem that
    \begin{align*}
    1&=\nu(\11_D)=
    \theta_{0,S}^{-n} \EE_\nu(\11_D(X_n)\11_{n< \tau_\d})\\
    &=\nu\left(\theta_{0,S}^{-n} n^{-\aj_S(\cdot)} \EE_\cdot(\11_D(X_n)\11_{n< \tau_\d})\right)\xrightarrow[n\to+\infty]{} \nu(\eta_S),
    \end{align*}
    which shows that $\nu(\eta_S)=1$. 
    
    Finally, integrating~\eqref{eq:eta} with respect to $\nu$ and using the fact that $j_S(x)=0$ $\nu(\mathrm dx)$-almost surely, we deduce that, for all $f\in L^\infty(W_S)$,
    \begin{align*}
    \left|\theta_{0,S}^{-n}  \EE_\nu(f(X_n)\11_{n< \tau_\d})-\sum_{i\in I_S} \nu(\eta_{S,i})\nu_{S,i}(f)\right|\leq \alpha_{S,n} \nu(W_S)\,\|f\|_{W_S}.
    \end{align*}
    Since $\theta_{0,S}^{-n}\,\EE_\nu(f(X_n)\11_{n< \tau_\d})=\nu(f)$ and $\nu(W_S)<+\infty$, we deduce that 
    $\nu=\sum_{i\in I_S}\nu(\eta_{S,i})\nu_{S,i}$. This concludes the proof of the proposition.
\end{proof}

In Assumption~(A), the $\nu_{S,i}$ do not need to be quasi-stationary distributions. However, we will see in the results of
Section~\ref{sec:finite-classes} that this is typically the case if the set $I_S$ and the measures $\nu_{S,i}$ are defined correctly.
In the following corollary, we consider a special situation where this holds true.

\begin{cor}
    \label{cor:corQSD}
        Assume that Assumption~(A) holds true  and that there exists a measurable partition $D=N\cup\left(\bigcup_{i\in I_S} M_i\right)$ such that for all $i\in I_S$, and all $x\in M_i$, we have $\nu_{S,i}(M_i)=1$, $j_S(x)=0$, $\eta_{S,i}(x)>0$, and, for all        
        $i\neq j\in I_S$ and $y\in M_j$, we have $\eta_{S,i}(y)=0$. Then the quasi-stationary distributions in $\mathcal M_+(W_S)$ with absorption parameter $\theta_{0,S}$ are exactly the convex combinations of the probability measures $\nu_{S,i}$. Similarly, the quasi-stationary distributions $\nu$ in $\mathcal M_+(W_S)$ such that $\nu(D\setminus N)>0$ are exactly the convex combinations of the probability measures $\nu_{S,i}$.
\end{cor}

\begin{rem}
    Observe that the set $N$ is $\sum_{i\in I_S}\nu_{S,i}$-negligible, but is typically non-empty since it contains all the points
    with $j_S>0$ (as we will see in the next sections, $j_S$ may not be identically zero in reducible state spaces).
    \erem
\end{rem}

\begin{proof}
    Proposition~\ref{prop:QSDbis} entails that any quasi-stationary distribution in $\mathcal M_+(W_S)$ with absorption parameter $\theta_{0,S}$ is a convex combination of the probability measures $\nu_{S,i}$. Reciprocally, for any $x\in M_i$, we have $\eta_S(x)=\eta_{S,i}(x)>0$ and hence, according to~\eqref{eq:eta},
    \[
    \frac{1}{\eta_{S}(x)}\sum_{i\in I_S}\eta_{S,i}(x)\nu_{S,i}=\nu_{S,i}
    \]
    is the limit  when $n\to+\infty$ of the conditional distribution $\mathbb P_x(X_n\in\cdot\mid n<\tau_\d)$, hence it is a quasi-limiting distribution and thus a  quasi-stationary distribution in $\mathcal M_+(W_S)$. Moreover $\nu_{S,i}$ satisfies $\nu_{S,i}(M_i)=1$ and, since $\eta_{S,i}>0$ on $M_i$,  $\nu_{S,i}(\eta_S)>0$. Proposition~\ref{prop:QSDbisrec} entails that the absorption parameter of $\nu_{S,i}$ is $\theta_{0,S}$. In particular, for any convex combination $\nu=\sum_{i\in I_S}\lambda_i \nu_{S,i}$, we have $\nu\in\mathcal M_+(W_S)$ and
    \begin{align*}
    \mathbb P_\nu(X_n\in\cdot)&=\sum_{i\in I_S}\lambda_i \mathbb P_{\nu_{S,i}}(X_n\in\cdot)=\sum_{i\in I_S}\lambda_i \theta_{0,S}^n \nu_{S,i}= \theta_{0,S}^n \nu,
    \end{align*}
    which implies that $\nu$ is a quasi-stationary distribution with absorption parameter $\theta_{0,S}$.
    
    To conclude the proof, we simply observe that any quasi-stationary distribution $\nu$ in $\mathcal M_+(W_S)$ such that $\nu(D\setminus N)>0$ satisfies $\nu(\eta_S)>0$ and hence, according to Proposition~\ref{prop:QSDbisrec}, its absorption parameter is $\theta_{0,S}$.
\end{proof}

We conclude this section with properties on the measures 
$\sum_i \eta_{S,i}(x) \nu_{S,i}$ and on $\eta_S$. 

\begin{prop}
    \label{prop:mainbis}
    Under Assumption (A):
    \begin{enumerate}[(i)]
        \item For all $x_*\in D$ such that $\eta_S(x_*)>0$,  $\frac{1}{\eta_{S}(x_*)}\sum_{i\in I_S}\eta_{S,i}(x_*)\nu_{S,i}$ is a
          quasi-stationary distribution with absorption parameter $\theta_{0,S}$. In addition, there exists $x\in D$ such that $\eta_S(x)>0$ and $j_S(x)=0$.
        \item The function $\eta_S$ satisfies, for all all $x\in D$,
        \begin{align*}
        \label{eq:prop-poly-2a}
        \EE_x \left[\eta_S(X_n)\11_{\aj_S(X_n)=\aj_S(x)}\right]=\theta_{0,S}^n\eta_S(x).
        \end{align*} 
        \item  For all $n\geq 0$ and all positive measure $\mu\in{\cal M}_+(W_S)$ such that $\mu(\eta_S)>0$  and $\mu(n^{j_S(\cdot)}W_S)<+\infty$, we have
         \begin{equation}
        \label{eq:prop-poly-2b}
        \left\| \PP_\mu\left(X_n\in\cdot\mid n<\tau_\d\right)-\frac{\sum_{i\in I_S}\mu\left(n^{\aj_S(\cdot)}\eta_{S,i}\right)\nu_{S,i}}{\sum_{i\in I_S}\mu\left(n^{\aj_S(\cdot)}\eta_{S,i}\right)}\right\|_{TV}\leq
        2\alpha_{S,n}\frac{\mu\left(n^{\aj_S(\cdot)}W_S\right)}{\mu\left(n^{\aj_S(\cdot)}\eta_S\right)}.
        \end{equation}
        \item  For all measure $\mu\in{\cal M}(W_S)$ and all $f\in L^\infty(W_S)$, we have
        \begin{multline}
        \label{eq:prop-poly-2c}
        \left|\theta_{0,S}^{-n}n^{-j_S(|\mu|)}\EE_\mu\left(f(X_n)\11_{n<\tau_\d}\right)-\sum_{i\in I_S}\mu(\11_{\aj_S(\cdot)=\aj_S(|\mu|)}\eta_{S,i})\nu_{S,i}(f)\right|\\
        \leq \left(\alpha_{S,n}+\frac{1}{n}\left\|\sum_{i\in I_S}\eta_{S,i}\nu_{S,i}(W_S)\right\|_{W_S}\11_{j_S(|\mu|)\geq 1}\right)\,|\mu|\left(W_S\right)\|f\|_{W_S},
        \end{multline}
    for all $n\geq 1$.
    \end{enumerate}
\end{prop}

%

We start with a preliminary lemma. Under Assumption~(A), we have, because of~\eqref{eq:prop12} in Proposition~\ref{prop:propajs1}, for any $\ell\geq 0$,
\[
G_\ell:=\left\{\mu\in\mathcal{M}(W_S),\aj_S(|\mu|)\leq\ell\right\}=\left\{\mu\in\mathcal{M}(W_s),\aj_S(x)\leq\ell\ |\mu|(\mathrm dx)\text{-ae}\right\}.
\]
Under Assumption~(A), the vector space $G_\ell$, endowed with the norm $\|\cdot\|_{W_S}$, is a Banach space.


{\color{black} 
	\begin{lem}
	\label{lem:technicbis}
	Assume that Assumption~(A) holds true and fix $\ell\geq 0$. Then the operator $\tS:G_\ell\to G_\ell$ defined by $\tS
	\mu=\theta_{0,S}^{-1}\mu S_1$ is a bounded linear operator and satisfies 
	\begin{equation}
	\label{eq:eta-H1-poly-bis}
		\left\| n^{-\ell} \tS^n\mu-E_\tS \mu\right\|_{W_S}\leq \alpha_{\tS,n}\,\|\mu\|_{W_S},\ \forall n\geq 1,\ \mu\in G_\ell,
	\end{equation}
	 with
	\[ 
	E_\tS\mu=\sum_{i\in I_S}\mu\left(\11_{\aj_S(\cdot)=\ell}\eta_{S,i}\right)\nu_{S,i}
	\]
	and
	\[
	\alpha_{\tS,n}=\alpha_{S,n}+\frac{1}{n}\left\|\sum_{i\in I_S}\eta_{S,i}\nu_{S,i}(W_S)\right\|_{W_S}\11_{\ell\geq 1},
	\]
	where in addition $E_\tS$ is a bounded linear operator on $G_\ell$ and $\alpha_{\tS,n}\to 0$ when $n\to +\infty$.
\end{lem}

\begin{proof}[Proof of Lemma~\ref{lem:technicbis}]
	We have $\sum_{i\in I_S} \eta_{S,i}\nu_{S,i}(W_S)\in L^\infty(W_S)$ by assumption.
	It follows from~\eqref{eq:eta} with $n=1$ that  $\mathcal{M}(W_S)$ is stable under $\tS$. Therefore, the stability of $G_\ell$ under $\tS$ is a consequence of~\eqref{eq:prop13} in Proposition~\ref{prop:propajs1}.
	Then, for all $\mu\in G_\ell$ and $f\in L^\infty(W_S)$ such that $\|f\|_{W_S}\leq 1$, we have, using Assumption~(A),
	\begin{multline*}
		\left|n^{-\ell}(\tS^n\mu)(f)-(E_\tS \mu)(f)\right|
		=\left|n^{-\ell}\theta_{0,S}^{-n}\mu S_n
		f-\sum_{i\in I_S}\mu\left(\11_{\aj_S(\cdot)=\ell}\eta_{S,i}\right)\nu_{S,i}(f)\right| \\ 
		\begin{aligned}
			& \leq \left|\theta_{0,S}^{-n}\mu\left(n^{-\aj_S(\cdot)} \11_{\aj_S(\cdot)=\ell} S_n
			f\right)-\sum_{i\in I_S}\mu\left(\11_{\aj_S(\cdot)=\ell}\eta_{S,i}\right)\nu_{S,i}(f)\right| \\
			&
			\qquad+\frac{1}{n}\theta_{0,S}^{-n}|\mu|\left(n^{-\aj_S(\cdot)} \11_{\aj_S(\cdot)\leq \ell-1} S_n W_S\right) \\ &
			\leq\alpha_{S,n}|\mu|\left(\11_{\aj_S(\cdot)=\ell}W_S\right)+\frac{1}{n}\left[\sum_{i\in I_S}|\mu|\left(\11_{\aj_S(\cdot)\leq\ell-1}\eta_{S,i}\right)\nu_{S,i}(W_S)\right.
			\\ & \qquad \left. +\alpha_{S,n}|\mu|\left(\11_{\aj_S(\cdot)\leq \ell-1}W_S\right)\right]\\
			& \leq \alpha_{S,n} |\mu|(W_S)+\frac{1}{n}|\mu|\left(\sum_{i\in I_S}\eta_{S,i}\nu_{S,i}(W_S)\right)\11_{\ell\geq 1}\\
			&\leq \alpha_{S,n} |\mu|(W_S) + \frac{1}{n}\left\|\sum_{i\in I_S}\eta_{S,i}\nu_{S,i}(W_S)\right\|_{W_S}\11_{\ell\geq 1} |\mu|(W_S).
		\end{aligned}
	\end{multline*}
	This implies~\eqref{eq:eta-H1-poly-bis}.
  Finally, since $G_\ell$ is a closed subset
	of the Banach space $\mathcal{M}(W_S)$, we deduce that $E_\tS\mu\in G_\ell$ for all $\mu\in G_\ell$ and the fact that $E_\tS$
        is a bounded operator on $G_\ell$ follows from~\eqref{eq:eta-H1-poly-bis}.
\end{proof}


%
In the following lemma, we let $\VERT\cdot\VERT$ denote the operator norm. 



\medskip

\begin{lem}
	\label{lem:fp-qsd-poly}
	Assume that~(A) holds, fix $\ell\geq 0$ and let $\tS:G_\ell\to G_\ell$ be defined as in Lemma~\ref{lem:technicbis}. Then,  
	\begin{enumerate}[(i)]
		\item for all $n\geq 1$,
		\[\VERT \tS^n \VERT\leq (\alpha_{\tS,n}+\VERT E_\tS\VERT)\, n^{\ell};\]
		\item if $\mu$ is an eigenvector of $\tS$ associated to $1$, then $E_\tS \mu=\11_{\ell=0}\mu$;
		\item  we have $\tS E_\tS=E_\tS \tS=E_{\tS}$; in particular, $E_\tS$ takes its values in the vector space of eigenvectors of $\tS$ associated to $1$.
	\end{enumerate}
\end{lem}

\begin{proof}[Proof of Lemma~\ref{lem:fp-qsd-poly}]
	Note that, according to Lemma~\ref{lem:technicbis}, inequality\eqref{eq:eta-H1-poly-bis} holds true.
	The first and second assertions are thus immediate. 	
	For the third one, we observe that
	\begin{multline*}
		\VERT {\tS}E_{\tS}-E_{\tS}\VERT \leq \VERT n^{-\ell}{\tS}^{n+1}-{\tS}E_{\tS}\VERT \\
		+\VERT (n+1)^{-\ell}{\tS}^{n+1}-n^{-\ell}{\tS}^{n+1}\VERT+\VERT (n+1)^{-\ell} {\tS}^{n+1}-E_{\tS}\VERT
	\end{multline*}
	where
	\begin{align*}
		\VERT n^{-\ell}{\tS}^{n+1}-{\tS}E_{\tS}\VERT &=\VERT {\tS}(n^{-\ell}{\tS}^n-E_{\tS})\VERT \leq  \VERT {\tS}\VERT\,\alpha_{{\tS},n}
	\end{align*}
	and
	\begin{align*}
		\VERT (n+1)^{-\ell}{\tS}^{n+1}-n^{-\ell}{\tS}^{n+1}\VERT & \leq \left(\frac{(n+1)^{\ell}}{n^{\ell}}-1\right)\VERT
                                                                           (n+1)^{-\ell}\tS^{n+1}\VERT \\ & \leq
                                                                                                            \left(\frac{(n+1)^{\ell}}{n^{\ell}}-1\right)
                                                                                                            (\alpha_{\tS,n}+\VERT
                                                                                                            E_\tS \VERT)
	\end{align*}
	by~(i), and
	\begin{align*}
		\VERT (n+1)^{\ell} {\tS}^{n+1}-E_{\tS}\VERT \leq \alpha_{{\tS},n+1},
	\end{align*}
	so that $\VERT {\tS}E_{{\tS}}-E_{\tS} \VERT \to 0$ when $n\to+\infty$, and hence 
	\[
	{\tS}E_{\tS} =E_{\tS} .
	\]
	Similarly, we have
	\begin{multline*}
		\VERT E_{{\tS}}{\tS}-E_{\tS}\VERT \leq \VERT n^{-\ell}{\tS}^{n+1}-E_{\tS} {\tS}\VERT \\
		+\VERT (n+1)^{-\ell}{\tS}^{n+1}-n^{-\ell}{\tS}^{n+1}\VERT +\VERT (n+1)^{\ell} {\tS}^{n+1}-E_{\tS}\VERT,
	\end{multline*}
	where
	\begin{align*}
		\VERT n^{-\ell}{\tS}^{n+1}-E_{\tS} {\tS}\VERT \leq \VERT {\tS}\VERT \,\VERT n^{-\ell}{\tS}^n-E_{\tS}\VERT\leq \alpha_{{\tS},n},
	\end{align*}
	and the other terms go to $0$ as in the previous case. Hence, we deduce that
	\[
		E_{\tS} {\tS}=E_{\tS} .\qedhere
	\]
\end{proof}
 }

\begin{proof}[Proof of Proposition~\ref{prop:mainbis}]
   {\color{black} \noindent \textit{Proof of (i).} Fix $x_*\in D$ such that $\eta_S(x_*)>0$ and let $\ell_*=\aj_S(x_*)$. 
    According to
    Lemma~\ref{lem:fp-qsd-poly} with $\ell=\ell_*$, the operator \[
    E_\tS\mu=\sum_{i\in I_S}\mu\left(\11_{\aj_S(\cdot)=\ell_*}\eta_{S,i}\right)\nu_{S,i}\] on $G_{\ell_*}$ satisfies
    \[
    (E_\tS \tS)\delta_{x}=(\tS E_\tS)\delta_{x}=E_\tS\delta_{x},\ \text{for all $x\in D$ such that $\aj_S(x)\leq\ell_*$.}
    \]
     This means that, for all $x\in D$ such that $\aj_S(x)\leq\ell_*$,
    \begin{align}
    \label{eq:equseful1bis}
    \theta_{0,S}^{-1}\sum_{i\in I_S}\delta_{x}  S_1\left(\11_{\aj_S(\cdot)=\ell_*}\eta_{S,i}\right)\nu_{S,i}
    &=\11_{\aj_S(x)=\ell_*}\theta_{0,S}^{-1}\sum_{i\in I_S} \eta_{S,i}(x)\nu_{S,i}S_1 \\
    &=\11_{\aj_S(x)=\ell_*}\sum_{i\in I_S} \eta_{S,i}(x)\nu_{S,i}.\label{eq:equseful1bis2}
    \end{align}
    Since $j_S(x_*)=\ell_*$, we deduce from the equality between the right-hand-side of~\eqref{eq:equseful1bis} and~\eqref{eq:equseful1bis2} that \[\nu:=\frac{1}{\eta_S(x_*)}\sum_{i\in I_S} \eta_{S,i}(x_*)\nu_{S,i}\in\mathcal M_+(W_S)\] is a quasi-stationary distribution for
    $S_1$ with absorption parameter $\theta_{0,S}$. In addition, according to Proposition~\ref{prop:QSDbis}, we have $\nu(\eta_S)=1$ and $\nu(\11_{j_S(\cdot)=0})=1$, so that 
    $\nu(\11_{j_S(\cdot)=0}\eta_S)=1$. Therefore, there exists $x\in D$ such that $\eta_S(x)>0$ and $\aj_S(x)=0$.
    
    \medskip\noindent \textit{Proof of (ii).} Fix $x\in D$. Then, applying as above Lemma~\ref{lem:fp-qsd-poly} with $\ell=\aj_S(x)$ instead of $\ell_*$, we obtain~\eqref{eq:equseful1bis} and~\eqref{eq:equseful1bis2} with $\ell_*$ replaced by $\aj_S(x)$. Integrating on both sides the test function $\11_D$, this
    implies that
    \[
    \eta_S(x)=\sum_{i\in I_S} \eta_{S,i}(x)= \theta_{0,S}^{-1}\sum_{i\in I_S}\delta_{x}  S_1\left(\11_{\aj_S(\cdot)=\aj_S(x)}\eta_{S,i}\right)=\theta_{0,S}^{-1}\delta_x S_1\left(\11_{\aj_S(\cdot)=\aj_S(x)}\eta_S\right).
    \]
    Hence
    \begin{align*}
    \E_x\left[\eta_S(X_1)\11_{\aj_S(X_1)=\aj_S(x)}\right]=\theta_{0,S}\eta_S(x).
    \end{align*}
    We deduce that, for all $n\geq 1$, $\mathbb P_x$-almost surely,
    \begin{align*}
    \E_{X_{n-1}}\left[\eta_S(X_1)\11_{\aj_S(X_1)=\aj_S(x)}\right]\11_{j_S(X_{n-1})=j_S(x)}=\theta_{0,S}\eta_S(X_{n-1})\11_{j_S(X_{n-1})=j_S(x)}.
    \end{align*}
    Taking the expectation and using the Markov property, we deduce that
    \begin{align*}
    \E_x\left[\eta_S(X_{n})\11_{\aj_S(X_n)=\aj_S(x)}\11_{j_S(X_{n-1})=j_S(x)}\right]=\theta_{0,S}\mathbb E_x\left[\eta_S(X_{n-1})\11_{j_S(X_{n-1})=j_S(x)}\right].
    \end{align*}
    Because of the last assertion of Proposition~\ref{prop:propajs1}, we deduce that
    $\{\aj_S(X_{n})=\aj_S(x)\}=\{\aj_S(X_{n})=\aj_S(X_{n-1})=\cdots=\aj_S(x)\}$ up to a $\mathbb P_x$-negligible event, so that
    \begin{align*}
    \E_x\left[\eta_S(X_{n})\11_{\aj_S(X_n)=\aj_S(x)}\right]=\theta_{0,S}\mathbb E_x\left[\eta_S(X_{n-1})\11_{j_S(X_{n-1})=j_S(x)}\right].
    \end{align*}
    Then the property (ii) follows by induction.
    
    \medskip\noindent \textit{Proof of (iii).} Fix $n\geq 1$ and let $\mu\in{\cal M}_+(W_S)$  such that $\mu(\eta_S)>0$ and $\mu(n^{j_S(\cdot)}W_S)<+\infty$. Integrating~\eqref{eq:eta} with
    respect to the measure
    $n^{\aj_S(x)}\,\mu(\dd x)$
    we obtain, for all $f\in L^\infty(W_S)$ such that $\|f\|_{W_S}\leq 1$,
    \begin{equation}
    \left|\theta_{0,S}^{-n}\EE_\mu(f(X_n)\11_{n<\tau_\d})-\sum_{i\in I_S}\mu\left(n^{\aj_S(\cdot)}\eta_{S,i}\right)\nu_{S,i}(f)\right|\leq \alpha_{S,n}\,\mu(n^{\aj_S(\cdot)}W_S).
    \end{equation}
    Note that $W_S\geq 1$ (by assumption) and hence this inequality also applies to $f\equiv 1$, which will be used just afterward.
    
    Then we have for all measurable function $f:D\to \mathbb R$ bounded by $1$
    \begin{multline*}
    \left|\frac{\EE_\mu f(X_n)\11_{n<\tau_\d}}{\PP_\mu(n<\tau_\d)}-\frac{\sum_{i\in I_S}\mu\left(n^{\aj_S(\cdot)}\eta_{S,i}\right)\nu_{S,i}(f)}{\sum_{i\in I_S}\mu\left(n^{\aj_S(\cdot)}\eta_{S,i}\right)}\right| \\
    \begin{aligned}
    &\leq\frac{\EE_\mu |f(X_n)|\11_{n<\tau_\d}}{\PP_\mu(n<\tau_\d)}\,\frac{|\sum_{i\in I_S}\mu\left(n^{\aj_S(\cdot)}\eta_{S,i}\right)-\theta_{0,S}^{-n}\PP_\mu(n<\tau_\d)|}{\sum_{i\in I_S}\mu\left(n^{\aj_S(\cdot)}\eta_{S,i}\right)} \\ 
    &\qquad \qquad +\frac{|\theta_{0,S}^{-n}\EE_\mu
        f(X_n)\11_{n<\tau_\d}-\sum_{i\in I_S}\mu\left(n^{\aj_S(\cdot)}\eta_{S,i}\right)\nu_{S,i}(f)|}{\sum_{i\in I_S}\mu\left(n^{\aj_S(\cdot)}\eta_{S,i}\right)} \\ 
    & \leq
    2\alpha_{S,n}\frac{\mu\left(n^{\aj_S(\cdot)}W_S\right)}{\mu\left(n^{\aj_S(\cdot)}\eta_S\right)},
    \end{aligned}
    \end{multline*}
    which concludes the proof.
    
    \medskip\noindent\textit{Proof of (iv).} Let $\mu$ be a measure in $\mathcal M(W_S)$. The property (iv) is an immediate consequence of Lemma~\ref{lem:technicbis} with $\ell=j_S(|\mu|)$. }
\end{proof}

\section{Quasi-stationary distributions in reducible state spaces with two successive sets}

\label{sec:three-sets}

We start our study of the quasi-stationary distribution for reducible processes by focusing on cases where the state space can be separated into two successive
classes. This is the generic situation that can be used iteratively to
treat more complicated cases (see Section~\ref{sec:finite-classes}).

We consider a discrete time Markov process $(X_n,n\in\ZZ_+)$ evolving in a measurable set $D\cup\{\d\}$ with absorption at $\d\notin
D$ at time $\tau_\d$, and sub-Markovian semigroup $(S_n)_{n\in\ZZ_+}$. 
We assume  that the transition probabilities of $X$ satisfy
the structure displayed in Figure~\ref{fig:trangraph}: there is a measurable partition $\{D_1,D_2\}$ of $D$ such that the process
starting from $D_1$ can access $D_1\cup D_2\cup\{\d\}$ and the process starting from $D_2$ can only access $D_2\cup\{\d\}$. More formally, we assume that $\P_x(T_{D_1}=+\infty)=1$ for all $x\in D_2$, where we denote,
for any measurable set $A\subset D$, $ T_A=\inf\{n\in\ZZ_+,\ X_n\in A\}$.

We denote by $(P_n)$ the sub-Markovian semigroup of the process $X$ restricted to $D_1$, by $(R_n)$ the sub-Markovian
semigroup of the processes $X$ restricted to $D_2$ and by $Q$ the transition kernel from $D_1$ to $D_2$ for $X$. More formally, for
all measurable $f:D_1\to [0,+\infty)$ and $g:D_2\to[0,+\infty)$, for all $x\in D_1$ and $y\in D_2$, we define
\[
P_n f(x)=\EE_x(f(X_n)),\  R_ng(y)=\EE_y(g(X_n))\ \text{and}\  Qg(x)=\EE_x(g(X_1)).
\]
Note that, due to our notational convention about extensions of functions by 0 outside of their domain, the previous definitions mean
\[
P_n f(x)=\EE_x(f(X_n)\11_{n<T_{D_2\cup\d}}),\  R_ng(y)=\EE_y(g(X_n)\11_{n<\tau_\d})\ \text{and}\  Qg(x)=\EE_x(g(X_1)\11_{X_1\in D_2}).
\]
\begin{figure}
    \center {
        
        \def\svgwidth{0.7\linewidth}
        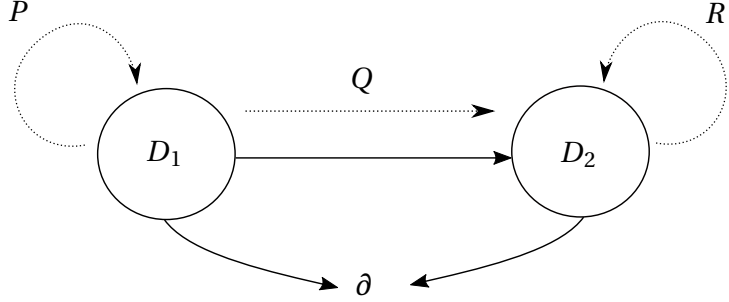}
    \caption{Transition graph displaying the relations between the sets $D_1$, $D_2$ and $\d$. The dashed lines indicate the domains and co-domains of the sub-Markov kernels $P,Q,R$.	\label{fig:trangraph}}
\end{figure}
{\color{black} In the rest of this section, the constants $\theta_{0,P}$ and $\theta_{0,R}$ denote respectively the exponential convergence parameters of the semigroups $(P_n)_{n\geq 0}$ and $(R_n)_{n\geq 0}$.}

\medskip

{\color{black} We will consider three situations. In the first one, we have $\theta_{0,P}>\theta_{0,R}$, so that the process evades $D_2$ at a strictly higher pace than it evades $D_1$, in which case we say that $D_1$ is a source. In the second one, we have $\theta_{0,P}<\theta_{0,R}$, so that the process evades $D_2$ at a strictly lower pace than it evades $D_1$, in which case we say that $D_1$ is a sink. In the third one, we have $\theta_{0,P}=\theta_{0,R}$, so that the process evades both $D_1$ and $D_2$ at the same pace, in which case we say that $D_1$ is a critical sink. As we will see, in the first situation, the assymptotic distribution of the process starting from $D_1$ and conditionned not to reach $\partial$ charges $D_1$, while, in the second and third situations, it only charges $D_2$.
	
	In order to prove this, we start by stating abstract results on the polynomial decay for upper triangular matrix of linear operators in Section~\ref{sec:app-2_new}. We then proceed to the proof of our probabilistic results in Section~\ref{sec:D1source} for the first case (where $\theta_{0,P}>\theta_{0,R}$), in Section~\ref{sec:D1sink} for the second case (where $\theta_{0,P}<\theta_{0,R}$), and in Section~\ref{sec:D1criticalsink} for the third case (where $\theta_{0,P}=\theta_{0,R}$).
}

\subsection{Polynomial decay for upper triangular matrix of linear operators}
\label{sec:app-2_new}

Let $B_1$ and $B_2$ be two Banach spaces, and ${\tP}:B_1\to B_1$, ${\tQ}:B_1\to B_2$, ${\tR}: B_2\to B_2$ three bounded operators. We define the Banach space $B$ as the direct sum of $B_1$ and $B_2$
and  consider the operator ${\tS}={\tP}+{\tQ}+{\tR}$ on $B$, where ${\tP}_{\vert B_2}={\tQ}_{\vert B_2}={\tR}_{\vert B_1}=0$. Formally, ${\tS}$ can be represented as the following upper triangular matrix of linear operators:
\[
{\tS}=\begin{bmatrix}
	{\tP}&{\tQ}\\
	0&{\tR}
\end{bmatrix},
\]
so that
\[
{\tS}^n=\begin{bmatrix}
	{\tP}^n&\sum_{k=1}^n {\tR}^{n-k}{\tQ}{\tP}^{k-1}\\
	0&{\tR}^n
\end{bmatrix}.
\]
{\color{black} Note that the configuration of the operator matrix $\tS$ corresponds to the configuration of the transition kernel between the sets $D_1$ and $D_2$: $\tP$ is related to the kernel from $D_1$ to itself, $\tQ$ to the kernel from $D_1$ to $D_2$ and $\tR$ is related to the kernel from $D_2$ to itself. This will be made precise in the proofs of the next sections.}

The study of the spectrum of such upper triangular matrices of linear operators over a Banach space has already been considered in the literature, see for instance~\cite{BarraaBoumazgour2003,BenhidaZeroualiEtAl2005,Barnes2005,Zhang2013} and references therein. In the following propositions, we are interested in the polynomial decay of the operator ${\tS}$, which is related to the algebraic multiplicity of its leading eigenvalue.

We are interested in the following property, which is related to Assumption~(A) and actually already appeared in~\eqref{eq:eta-H1-poly-bis}, and the way it translates from $\tP$ and $\tR$ to $\tS$.

\medskip \noindent \textbf{Assumption (H).} There exists a bounded linear operator $E_\tP$ on $B$ and $\tJ_\tP\in\RR_+$ such that, for all $x\in B$ and all $n\geq 1$,
\begin{equation}
    \label{eq:eta-H1-poly}
    \left| n^{-\tJ_\tP} \tP^nx-E_\tP x\right|\leq \alpha_{\tP,n}\,|x|,\ 
\end{equation}
where $(\alpha_{\tP,n})_{n\in\N}$ is a numerical sequence which converges to $0$ when $n\to+\infty$. 

For all $n\geq 0$, we set $\gamma_n=\left\VERT {\tR}^n\right\VERT$, $\Gamma_n=\sum_{k\geq n}\gamma_k$, $\theta_n=\left\VERT {\tP}^n\right\VERT$ and $\Theta_n=\sum_{k\geq n}\theta_k$. We consider :
\begin{itemize}
\item the case $\Gamma_0<+\infty$ in Proposition~\ref{prop:case1new} {\color{black}(this will correspond to the situation where the
    process escapes $D_2$ at a strictly higher pace than it evades $D_1$, that is $D_1$ is a source, see Section~\ref{sec:D1source})},
\item the case $\Theta_0<+\infty$ in Proposition~\ref{prop:case2new} {\color{black}(this will correspond to the situation where the
    process evades $D_2$ at a strictly lower pace than it evades $D_1$, that is $D_1$ is a sink, see Section~\ref{sec:D1sink})},
\item  the case $\Gamma_0=\Theta_0=+\infty$ in Proposition~\ref{prop:case3new} {\color{black}(this will correspond to the situation
    where the process evades both $D_1$ and $D_2$ at the same pace, that is $D_1$ is a critical sink, see Section~\ref{sec:D1criticalsink})}.
\end{itemize}

\begin{prop}
	\label{prop:case1new}
	Assume that $\Gamma_0=\sum_{n=0}^\infty\left\VERT {\tR}^n\right\VERT<+\infty$ and that the operator ${\tP}$ satisfies Assumption~(H). Then ${\tS}$ satisfies assumption (H) with 
	\[
	\tJ_{\tS}=\tJ_{\tP},\quad E_{\tS}=E_{\tP}+\sum_{\ell\geq 0} {\tR}^{\ell}{\tQ}E_{\tP},
	\]
	and 
	\[
	\alpha_{{\tS},n}= \alpha_{{\tP},n}+C \Gamma_n+C\,\sum_{k=0}^{n-1}\gamma_k\left(\alpha_{{\tP},n-k-1}+\frac{\tJ_{\tP}\,(k+1)}{n}\right),
	\]
	for some positive constant $C>0$ which does not depend on $n\geq 1$, and
	with the convention that $\alpha_{\tP,0}=1$.
\end{prop}

In the following proof, we will use repeatedly that, for all $n\geq 1$, $k\in \{0,n\}$ and $j\geq 0$,
\begin{align}
	\label{eq:ineq-used-several-times}
0\leq 	1-\left(\frac{n-k-1}{n}\right)^j\leq \frac{j\,(k+1)}{n}.
\end{align}

\begin{proof} 

	Fix $n\geq 1$ and $x\in B_1$. Then
	\begin{multline}
		|n^{-\tJ_{\tP}}{\tS}^nx - E_{\tS} x| \leq |n^{-\tJ_{\tP}}{\tP}^nx-E_{\tP}x|+\sum_{k=0}^{n-1} \left|n^{-\tJ_{\tP}} {\tR}^k{\tQ}{\tP}^{n-k-1}x-{\tR}^k{\tQ}E_{\tP}x\right|\\
		+\sum_{k=n}^{\infty} |{\tR}^k{\tQ}E_{\tP}x|.
		\label{eq:decomp}
	\end{multline}
	Using Assumption~(H) and the fact that $\tQ$ and $\tP$ are bounded operators, 
%
we deduce that 
	\begin{align*}
		|n^{-\tJ_{\tP}}{\tP}^nx-E_{\tP}x|+\left|\sum_{k=n}^{\infty} {\tR}^k{\tQ}E_{\tP}x\right|\leq \alpha_{{\tP},n}|x|+\VERT {\tQ} E_{\tP}\VERT \Gamma_n |x|.
	\end{align*}

	For the second term in the r.h.s. of~\eqref{eq:decomp}, we have for all $k\in\{0,\ldots,n-2\}$,
	\begin{align*}
		\left|n^{-\tJ_{\tP}} {\tR}^k{\tQ}{\tP}^{n-k-1}x-{\tR}^k{\tQ}E_{\tP}x\right|&\leq \left|1-\left(\frac{n-k-1}{n}\right)^{\tJ_{\tP}}\right|\,\VERT {\tR}^k{\tQ}\VERT\,\left|(n-k-1)^{-\tJ_{\tP}}{\tP}^{n-k-1}x\right|\\
		&\qquad\qquad + \VERT {\tR}^k{\tQ}\VERT\,\left|(n-k-1)^{-\tJ_{\tP}}{\tP}^{n-k-1}x-E_{\tP}x\right|\\
		&\leq \VERT {\tQ}\VERT\,\frac{\tJ_{\tP}(k+1)}{n}\,\gamma_k\,\left(\alpha_{{\tP},n-k-1}+\VERT
		E_{\tP}\VERT\right)\,|x|\\
		&\qquad\qquad+\VERT \tQ\VERT\gamma_k\alpha_{{\tP},n-k-1}\,|x|
	\end{align*}
	where we used~\eqref{eq:ineq-used-several-times}, Assumption~(H) and its immediate consequence
	\begin{align}
		\label{eq:usefull_new}
		\VERT \tP^n\VERT\leq (\alpha_{{\tP},n}+\VERT E_\tP\VERT)n^{\tJ_{\tP}}.
	\end{align}
  For $k=n-1$, we observe that 
	\[
	\left|n^{-\tJ_{\tP}} {\tR}^k{\tQ}{\tP}^{n-k-1}x-{\tR}^k{\tQ}E_{\tP}x\right|\leq 2(\VERT {\tQ}\VERT+\VERT {\tQ E_{\tP}}\VERT) \gamma_{n-1}\,|x|.
	\]
	
	Fiw now $x\in B_2$. Then ${\tS}^nx={\tR}^nx$ and $E_{\tP} x=0$, so that $|{\tS}^nx-E_{\tS} x|\leq \gamma_n|x|\leq \Gamma_n|x|$.
	
		We finally deduce that
	\begin{align*}
		\left|n^{-j_{\tS}}{\tS}^nx-E_{\tS}x\right|&\leq \alpha_{{\tS},n}\,|x|,
	\end{align*}
	where, for some constant $C>0$,
	\begin{align*}
		\alpha_{{\tS},n}=\alpha_{{\tP},n}&+C \Gamma_n+ C\,\sum_{k=0}^{n-1}\gamma_k\left(\alpha_{{\tP},n-k-1}+\frac{\tJ_{\tP}\,(k+1)}{n}\right),
	\end{align*}
	which converges to $0$ when $n\to+\infty$.
	
	
\end{proof}

%
%
%
%
%
%

\begin{prop}
	\label{prop:case2new}
	Assume that $\Theta_0=\sum_{n=0}^\infty\left\VERT {\tP}^n\right\VERT<+\infty$ and that the operator ${\tR}$ satisfies Assumption~(H). Then ${\tS}$ satisfies assumption (H) with
	\[
	\tJ_{\tS}=\tJ_{\tR},\quad E_{\tS}=E_{\tR}+\sum_{\ell\geq 0} E_{\tR} {\tQ}{\tP}^\ell,
	\]
	and
	\[
	\alpha_{{\tS},n}=\alpha_{{\tR},n}+ C\Theta_n+ C \sum_{k=0}^{n-1} \theta_{k}\left( \alpha_{{\tR},n-k-1}+\frac{\tJ_{\tR}\, (k+1)}{n}\right),
	\]
	for some positive constant $C>0$, which does not depend on $n\geq 1$, and
	with the convention that $\alpha_{\tR,0}=1$.
\end{prop}

\begin{proof}
	We have, for all $n\geq 1$ and all $x\in B$,
	\begin{align*}
		\left|n^{-\tJ_{\tS}}{\tS}^nx-E_{\tS}x\right|&\leq \left|n^{-\tJ_{\tR}}{\tR}^nx-E_{\tR}x\right|+n^{-\tJ_{\tR}}\left|\tP^n x\right|\\
		&\qquad\qquad+n^{-\tJ_{\tR}}\sum_{k=0}^{n-1} \left|{\tR}^{n-k-1}{\tQ}{\tP}^{k}x-(n-k-1)^{\tJ_{\tR}}E_{\tR} {\tQ}{\tP}^{k}x\right|\\
		&\qquad\qquad+\sum_{k=0}^{n-1} \left(1- \left(\frac{n-k-1}{n}\right)^{\tJ_{\tR}}\right)|E_{\tR}{\tQ}{\tP}^{k}x|\\
		&\qquad\qquad+\sum_{k=n}^\infty  |E_{\tR} {\tQ}{\tP}^{k}x|.
	\end{align*}
	Using Assumption~(H) for ${\tR}$ and the fact that ${\tQ}$ is a bounded operator, we deduce that the first three terms are bounded by
	\begin{multline*}
		\alpha_{{\tR},n}|x|+\theta_n|x|+\VERT {\tQ}\VERT \sum_{k=0}^{n-1} \alpha_{{\tR},n-k-1} \left(\frac{n-k-1}{n}\right)^{\tJ_R}\theta_{k} |x|
		\\ \leq \left(\alpha_{{\tR},n}{\color{black}+\theta_n}+\VERT {\tQ}\VERT \sum_{k=0}^{n-1} \alpha_{{\tR},n-k-1}\theta_{k}\right)\,|x|.
	\end{multline*}
	Using~\eqref{eq:ineq-used-several-times}, we deduce that the fourth and fifth terms are bounded by
	\begin{multline*}
		\sum_{k=0}^{n-1} \VERT E_{\tR} {\tQ}\VERT \left(1- \left(\frac{n-k-1}{n}\right)^{\tJ_{\tR}}\right)\theta_{k}|x|+\sum_{k=n}^\infty  \VERT E_{\tR} {\tQ}\VERT  \theta_{k}|x|\\
		\leq \VERT E_{\tR} {\tQ}\VERT \left(\sum_{k=0}^{n-1} \frac{\tJ_{\tR}\,(k+1)}{n}\,\theta_k+\Theta_n\right)\,|x|.
	\end{multline*}
	We finally deduce that
	\begin{align*}
		\left|n^{-j_{\tS}}{\tS}^nx-E_{\tS}x\right|&\leq \alpha_{{\tS},n}\,|x|,
	\end{align*}
	where, for some constant $C>0$,
	\[
	\alpha_{{\tS},n}= \alpha_{{\tR},n}+C \sum_{k=0}^{n-1} \alpha_{{\tR},n-k-1}\theta_{k}+ C\Theta_n+C\,\sum_{k=0}^{n-1} \frac{\tJ_{\tR}\,(k+1)}{n}\,\theta_k,
	\]
	{\color{black}which goes to $0$ when $n\to +\infty$.}
\end{proof}

\begin{prop}
	\label{prop:case3new}
	Assume ${\tP}$ and ${\tR}$ both satisfy Assumption~(H), with $\tJ_{\tP}= 0$. Then ${\tS}$ satisfies Assumption~(H) with
	\[
	\tJ_{\tS}=1+\tJ_{\tR},\quad E_{\tS}=\frac{1}{\tJ_{\tS}}E_{\tR}{\tQ}E_{\tP}
	\]
	and
	\[
		\alpha_{{\tS},n}= \frac{C}{n}\Big( 
		\tJ_{\tR}+ \sum_{k=0}^n \alpha_{{\tP},k}+\sum_{k=0}^{n} \alpha_{{\tR},n-k} \left(\frac{n-k}{n}\right)^{\tJ_{\tR}} \Big)
	\]
	for some positive constant $C$, which does not depend on $n\geq 1$, and with the convention that $\alpha_{\tP,0}=\alpha_{\tR,0}=1$. 
\end{prop}



\begin{proof}
	Using the fact that ${\tS}^n x={\tP}^nx+\sum_{k=1}^n {\tR}^{n-k}{\tQ}{\tP}^{k-1}x{\color{black}+\tR^n x}$, we deduce that
	\begin{align}
		\label{eq:maj}
		\left|n^{-\tJ_{\tS}}{\tS}^nx-E_{\tS}x\right| 
		& \leq |n^{-\tJ_{\tS}}{\tP}^nx| +|n^{-\tJ_{\tS}}{\tR}^nx|\notag\\ 
		&\qquad\qquad+n^{-\tJ_{\tS}}\sum_{k=1}^{n} \left|{\tR}^{n-k}{\tQ}{\tP}^{k-1}x-(n-k)^{\tJ_{\tR}}E_{\tR}{\tQ}{\tP}^{k-1}x\right|\notag\\
		&\qquad\qquad +n^{-\tJ_{\tS}}\sum_{k=1}^n (n-k)^{\tJ_{\tR}}\left|E_{\tR}{\tQ}{\tP}^{k-1}x -E_{\tR}{\tQ}E_{\tP}x\right|\notag\\
		&\qquad\qquad +\left|n^{-\tJ_{\tS}}\sum_{k=1}^n (n-k)^{\tJ_{\tR}}-\frac{1}{\tJ_{\tS}}\right|\left|E_{\tR}{\tQ}E_{\tP}x\right|
	\end{align}
	
	For the first two terms on the right hand side, we deduce from~\eqref{eq:usefull_new} applied to $\tP$ and $\tR$, {\color{black}and the fact that $\tJ_\tS\geq 1+\tJ_R$ and $\tJ_\tS \geq 1+\tJ_P$} that
	\begin{align}
		\label{eq:maj1}
		n^{-\tJ_{\tS}}|{\tP}^nx|+n^{-\tJ_{\tS}}|{\tR}^nx|\leq (\alpha_{{\tP},n}+\alpha_{{\tR},n}+\VERT E_{\tP}\VERT +\VERT E_{\tR}\VERT)\,n^{-1}\,|x|.
	\end{align}
	
	For the third term, we use that, for all $n\geq 1$, using again~\eqref{eq:usefull_new} and the boundedness of ${\tQ}$,
	\[
	|{\tQ}{\tP}^{n-1}x|\leq \VERT {\tQ}\VERT  (\alpha_{{\tP},n-1}+\VERT E_{\tP}\VERT) |x|.
	\]
	Hence,  using Assumption~(H) for ${\tR}$, we obtain, for all $k\geq 1$,
	\begin{align*}
		\left|{\tR}^{n-k}{\tQ}{\tP}^{k-1}x-(n-k)^{\tJ_{\tR}}E_{\tR}{\tQ}{\tP}^{k-1}x\right|
		&\leq \alpha_{{\tR},n-k} (n-k)^{\tJ_{\tR}} \VERT {\tQ}\VERT  (\alpha_{{\tP},k-1}+\VERT E_{\tP}\VERT) |x|.
	\end{align*}
	Thus, using again the fact that $\tJ_\tS\geq 1+\tJ_R$ and $\tJ_\tS \geq 1+\tJ_P$, 
	\begin{multline}
		n^{-\tJ_{\tS}}\sum_{k=1}^{n} \left|{\tR}^{n-k}{\tQ}{\tP}^{k-1}x-(n-k)^{\tJ_{\tR}}E_{\tR}{\tQ}{\tP}^{k-1}x\right| \\ \leq  \VERT
		{\tQ}\VERT\,|x|\, \left(\max_{k\geq 0}\alpha_{{\tP},k}+\VERT E_{\tP}\VERT\right)\,
		\frac{1}{n}\sum_{k=1}^{n} \alpha_{{\tR},n-k} \left(\frac{n-k}{n}\right)^{\tJ_{\tR}} .
		\label{eq:maj2}
	\end{multline}
	
	For the fourth term, we use Assumption~(H) for ${\tP}$ to derive (using again the fact that $\tJ_\tS\geq 1+\tJ_R$ and $\tJ_\tS \geq 1+\tJ_P$)
	\begin{align}
		\label{eq:maj3}
		n^{-\tJ_{\tS}}\sum_{k=1}^n (n-k)^{\tJ_{\tR}}\left|E_{\tR}{\tQ}{\tP}^{k-1}x -E_{\tR}{\tQ}E_{\tP}x\right|\leq \frac{\VERT E_{\tR} {\tQ}\VERT|x|}{n}\sum_{k=1}^n \alpha_{{\tP},k-1}(1-k/n)^{\tJ_{\tR}}.
	\end{align}
	
	Finally, using again the fact that $\tJ_\tS\geq 1+\tJ_R$ and $\tJ_\tS \geq 1+\tJ_P$, the fifth term in~\eqref{eq:maj} is bounded by
	\begin{align}
		&\VERT E_{\tR} {\tQ} E_{\tP}\VERT\left|n^{-\tJ_{\tS}}\sum_{k=1}^n (n-k)^{\tJ_{\tR}}-\frac{1}{\tJ_{\tS}}\right|\,|x|
		\nonumber\\
		&\qquad\qquad\qquad\leq \VERT E_{\tR} {\tQ} E_{\tP}\VERT \sum_{k=1}^n \left|\frac{1}{n}(1-k/n)^{\tJ_{\tR}}-\int_{(k-1)/n}^{k/n}(1-u)^{\tJ_{\tR}}\,du\right|\,|x|\notag\\
		&\qquad\qquad\qquad\leq  \VERT E_{\tR} {\tQ} E_{\tP}\VERT\frac{\tJ_{\tR}}{n}\,|x|.\label{eq:maj4}
	\end{align}
	
	Combining~\eqref{eq:maj} and the bounds~\eqref{eq:maj1},~\eqref{eq:maj2},~\eqref{eq:maj3},~\eqref{eq:maj4} ends the proof of Proposition~\ref{prop:case3new}.
\end{proof}

\subsection{\color{black}  Case where $D_1$ is a source ($\theta_{0,P}>\theta_{0,R}$)} 

\label{sec:D1source}

%
%
%
%


{\color{black} In this section, we consider the situation where the process evades $D_2$ at a strictly higher pace than it evades $D_1$. This is made precise by the following assumption, which will allow us to make use of Proposition~\ref{prop:case1new}.}

\medskip \noindent \textbf{Assumption (A1)} We have $j_{0,P}<+\infty$, the process $X$ restricted to $D_1$ satisfies Assumption~(A)  and there exists a measurable function $W_R:D_2\to[1,+\infty)$ such that, for some constants $\gamma\in[0,\theta_{0,P})$ and $c_1>0$, for all $x\in D_1$ and $y\in D_2$, 
\begin{equation}
  \label{eq:A1}
  \E_x(W_R(X_1))\leq W_P(x)\text{ and }\E_y(W_R(X_n))\leq c_1\gamma^n W_R(y),\ \forall n\geq 0.  
\end{equation}

\begin{rem}
	\label{rem:A1withconstants}
	Note that Assumption~(A) remains valid if  the function $W_S$ is multiplied by a positive constant. Hence, in the above assumption (A1)  the requirement $\E_x(W_R(X_1))\leq W_P(x)$ is actually equivalent to  $\E_x(W_R(X_1))\leq C\,W_P(x)$ for some positive constant $C>0$. 
	\erem
\end{rem}

\begin{rem}
	\label{rem:defWRWP}
	Possible candidates for $W_R\geq 1$ in Assumption~(A1) 
	are the exponential moment of
	exit times from $D_2$. 
	Indeed, if $W_R(y)=\E_y(\gamma^{-\tau_\d})$ is finite for all $y\in D_2$, then
	$\E_y(W_R(X_1))= \gamma W_R(y)$ for all $y\in D_2$. Indeed, we have, using the Markov property at time $1$ and the fact that, for all $y\in D_2$, $\tau_\d\geq 1$,
	\begin{align*}
		\E_y(W_R(X_1))=\E_y(\E_{X_1}(\gamma^{-\tau_\d}))=\E_y(\gamma^{-(\tau_\d+1)})=\gamma^{-1} W_R(y).
	\end{align*}
	
	\erem
\end{rem}

{\color{black} The following theorem states that Assumption~(A1) implies Assumption~(A), with explicit parameters.}

\begin{thm}
	\label{thm:A1}
		Assume that Assumption~(A1) holds true. Then $X$ satisfies Assumption~(A) with $W_S=W_P+W_R$.
	Moreover, we have $\theta_{0,S}=\theta_{0,P}$, $j_S=j_P$, and,  for all $i\in I_S=I_P$, $\eta_{S,i}\propto\eta_{P,i}$. In addition, there exists a constant $C>0$, independent of $x\in D$ and $n\in\NN$, such that
		\[
		\nu_{S,i}\propto \nu_{P,i}+\sum_{k\geq 0} \theta_{0,S}^{-k-1}\P_{\nu_{P,i}}\left(T_{D_2}=1,\ X_{k+1}\in\cdot\right),
		\]
		with inverse proportionality constant than for $\eta_{S,i}\propto\eta_{P,i}$,
		and 
		{\color{black}\[
		\alpha_{S,n}=C \sum_{k=0}^{n} \left(\frac{\gamma}{\theta_{0,P}}\right)^k \cdot \left(\alpha_{P,n-k}+j_{0,P}\frac{k}{n}\right),
              \]}
     for some positive constant $C>0$ which does not depend on $n$, and with the convention that $\alpha_{P,0}=1$.
\end{thm}

\begin{rem}
	\label{rem:slowconv}
	In the conclusion of the last theorem, if $j_P$ (resp. $j_R$) is not identically equal to 0, then the convergence rate of
        $\alpha_{S,n}$ to 0 is $O(1/n)$, even if $\alpha_{P,n}$ converge geometrically to $0$.
	\erem
\end{rem}

\begin{proof}[Proof of Theorem~\ref{thm:A1}] 
	We define 
	the linear operators $\tP:{\cal M}(W_P)\to {\cal M}(W_P)$, $\tQ:{\cal M}(W_P)\to {\cal M}(W_R)$ and $\tR:{\cal M}(W_R)\to
	{\cal M}(W_R)$ (these notations implicitely assume that $\mathcal{M}(W_P)\subset\mathcal{M}(D_1)$ and
	$\mathcal{M}(W_R)\subset\mathcal{M}(D_2)$) by
	\begin{align*}
		\tP \mu=\theta_0^{-1}\mu P_1,\quad \tQ \mu = \theta_0^{-1}\mu Q,\quad\text{and }\tR \mu = \theta_0^{-1}\mu R_1,
	\end{align*} 
	where $\theta_0=\theta_{0,P}$.
	{\color{black} Using Assumption (A1) we observe that all these operators are bounded. Our aim is to apply Proposition~\ref{prop:case1new} to $\tS:{\cal M}(W_S)\to {\cal M}(W_S)$, where ${\cal M}(W_S)\equiv {\cal M}(W_P)\oplus {\cal
		M}(W_R)$, with $W_S=W_R+W_P$ and $\tS=\tP+\tQ+\tR$.} Beware that $\tP,\tR,\tQ,\tS$ act on the left on $\mu$ while $P_n,R_n,Q,S_n$ act on the right, so that, for instance, $\tR \tP\mu=\theta_0^{-2}\mu P_1 R_1$.

	\medskip
	We define $B_2$ as the Banach space ${\cal M}(W_R)$ and observe that the operator $\tR:B_2\to B_2$ is bounded and
	\begin{align*}
		\sum_{n=0}^\infty \VERT \tR^n\VERT\leq  \sum_{n=0}^\infty\ \theta_{0,P}^{-n}\sup_{\mu\in B_2,\ |\mu|(W_R)=1} |\mu| R_n W_R.
	\end{align*}
	It follows from~\eqref{eq:A1} that
	\begin{align*}
		\sum_{n=0}^\infty \VERT \tR^n\VERT\leq  c_1\sum_{n=0}^\infty\ \theta_{0,P}^{-n} \gamma^n  \sup_{\mu \in B_2,\ |\mu|(W_R)=1} |\mu| (W_R) = \frac{c_1\theta_0}{\theta_{0,P}-\gamma}.
	\end{align*}
	Moreover, $\Gamma_n\leq c_1\frac{\gamma^n/\theta_{0,P}^{n-1}}{\theta_{0,P}-\gamma}$ and $\gamma_n\leq c_1\gamma^n/\theta_{0,P}^n$ (using the notations of Proposition~\ref{prop:case1new}).
	In particular, if $x\in D_2$, then~\eqref{eq:eta} holds true with $\eta_{S,i}(x)=0$ for all $i\in I_S=I_P$, $j_S(x)=0$, $W_S(x)=W_R(x)$ and $\alpha_{S,n}=c_1(\frac{\gamma}{\theta_{0,P}})^n$.

	From now on we assume that $x\in D_1$ and consider the vector space $B_1=\{\mu\in\mathcal M(W_P),\ j_P(|\mu|)\leq j_P(x)\}$. By
	Lemma~\ref{lem:technicbis}, the operator $\tP:B_1\to B_1$ satisfies~\eqref{eq:eta-H1-poly-bis} with $\ell=j_P(x)$, and hence Assumption~(H) with $\tJ_\tP=j_P(x)$ and
	\begin{align}
		\label{eq:Epmu1}
		E_\tP \mu=\sum_{i\in I_P}\mu(\11_{\aj_P(\cdot)=\aj_P(x)}\eta_{P,i})\nu_{P,i} 
	\end{align}
	and, using the fact that $\left\|\sum_{i\in I_P}\eta_{P,i}\nu_{P,i}(W_P)\right\|_{W_P}$ is finite,
	{\color{black}\[
	\alpha_{\tP,n}=C\left(\alpha_{P,n}+\frac{\11_{j_P(x)\geq 1}}{n}\right)
	\]
	for some constant $C>0$.}
	
	Note also that $\tQ:B_1\to B_2$ is a bounded operator by~\eqref{eq:A1}. As a consequence, according to Proposition~\ref{prop:case1new}, $\tS$ restricted to $B=B_1\oplus B_2$ also satisfies Assumption~(H)
	with $\tJ_\tS=j_P(x)$ and for all $\mu\in B_1$, 
	{\color{black}\begin{align*}
		E_\tS\mu &=E_\tP\mu +\sum_{k\geq 0} \tR^k\tQ E_\tP\mu\\
		&=\sum_{i\in I_P}\mu(\11_{\aj_P(\cdot)=\aj_P(x)}\eta_{P,i})\nu_{P,i}+\sum_{k\geq 0} \theta_{0,P}^{-k-1}\sum_{i\in I_P}\mu(\11_{\aj_P(\cdot)=\aj_P(x)}\eta_{P,i})\P_{\nu_{P,i}}\left(T_{D_2}=1,\ X_{k+1}\in\cdot\right)\\
		&=\sum_{i\in I_P}\mu(\11_{\aj_P(\cdot)=\aj_P(x)}\eta_{P,i})\left(\nu_{P,i}+\sum_{k\geq 0} \theta_{0,P}^{-k-1}\P_{\nu_{P,i}}\left(T_{D_2}=1,\ X_{k+1}\in\cdot\right)\right).
	\end{align*}}
	and
	\begin{align}
		\alpha_{\tS,n}&=\alpha_{{\tP},n}+ C \Gamma_n+C\,\sum_{k=0}^{n-1}\gamma_k\left(\alpha_{{\tP},n-k-1}+\frac{\tJ_{\tP}(k+1)}{n}\right)\nonumber\\
		& \leq C \sum_{k=0}^{n} \left(\frac{\gamma}{\theta_{0,P}}\right)^k \left(\alpha_{P,n-k}+\frac{\11_{j_P(x)\geq
			1}}{n+1-k}+j_P(x)\frac{k}{n}\right) \label{eq:xdependstep1}\\ &
		\leq \alpha_{S,n}:=C \sum_{k=0}^{n} \left(\frac{\gamma}{\theta_{0,P}}\right)^k \left(\alpha_{P,n-k}+j_{0,P}\frac{k}{n}\right)\nonumber
	\end{align}
	for some constant $C>0$ that may change from line to line, where we used $\11_{j_{0,P}\geq 1}\leq j_{0,P}$ and $\frac{1}{n-k+1}\leq\frac{k}{n}$.
	Using the fact that $\tS^n\mu=\theta_{0,P}^{-n}\mu S_n$ and taking $\mu=\delta_x$, we deduce that, for all $x\in D$ and all $f\in L^\infty(W_S)$,
	\begin{multline}
		\label{eq:tli1new}
		\left|n^{-j_P(x)}\theta_{0,P}^{-n}S_n f(x)-\sum_{i\in I_P}\eta_{P,i}(x)\left(\nu_{P,i}(f)+\sum_{k\geq 0} \theta_{0,P}^{-k-1} \E_{\nu_{P,i}}\left(\11_{T_{D_2}=1}f(X_{k+1})\right)\right)\right|\\
		\leq \alpha_{S,n}W_S(x)|f|.
	\end{multline}

	It only remains to prove that $j_S(x)=j_P(x)$ for all $x\in D$ (recall that under our convention $j_P$ is extended to $D_2$ by the value $0$). On the one hand, the definitions of $j_S$, $j_P$ and $S$ clearly imply that $j_S(x)\geq j_P(x)$ for all $x\in D$. On the other hand,  inequality~\eqref{eq:tli1new} implies that, for all $\varepsilon>0$, 
	\[
	\liminf_{n\to+\infty} n^{-(j_P(x)+\varepsilon)}\theta_{0,P}^{-n}S_n \11_D(x)=0,
	\]
	so that $j_S(x)\leq j_P(x)+\varepsilon$ for all $\varepsilon>0$, and hence $j_S(x)\leq j_P(x)$. This concludes the proof of Theorem~\ref{thm:A1}.
\end{proof}

    \subsection{\color{black} Case where $D_1$ is a sink ($\theta_{0,P}<\theta_{0,R}$)} 

\label{sec:D1sink}

%
%
%
%


{\color{black} In this section, we consider the situation where the process evades $D_2$ at a strictly lower pace than it evades $D_1$. This is made precise in the following assumption, which will allow us to make use of Proposition~\ref{prop:case2new}.}

\medskip \noindent \textbf{Assumption (A2)}  We have $j_{0,R}<+\infty$, the process $X$ restricted to $D_2$ satisfies Assumption~(A) and there exists a measurable function $W_P:D_1\to[1,+\infty)$ such that, for some constants $\gamma\in[0,\theta_{0,R})$ and $c_2>0$, for all $x\in D_1$, 
\begin{align}
    \label{eq:A2}
\E_x(W_R(X_1))\leq W_P(x)\text{ and }\E_x(W_P(X_n))\leq c_2 \gamma^nW_P(x),\ \forall n\geq 0.
\end{align}

{\color{black} We emphasize that Remarks~\ref{rem:A1withconstants} and~\ref{rem:defWRWP} (with $W_R$ and $D_2$ replaced by $W_P\geq 1$
  and $D_1$) also apply to Assumption~(A2). The following theorem states that Assumption~(A2) implies Assumption~(A), with explicit
  parameters. In this situation the limiting distribution of the process starting from $D_1$ only charges $D_2$.}

\medskip

\begin{thm}
    \label{thm:A2}
    Assume that Assumption~(A2) holds true. Then $X$ satisfies Assumption~(A) with $W_S=W_P+W_R$.
    Moreover, there exists a constant $C>0$, independent of $x\in D$ and $n\in\NN$, such that $\theta_{0,S}=\theta_{0,R}$ and, for all $x\in D$,  
    \[
    j_S(x)=\begin{cases}
        \max_{n\geq 0} j_R(\delta_x P_n Q)&\text{ if }x\in D_1,\\
        j_R(x)&\text{ if }x\in D_2
    \end{cases}
    \] 
    and for all $i\in I_S=I_R$,
    \[
    \eta_{S,i}(x)=
    \EE_x\left(\theta_{0,R}^{- T_{D_2}}\eta_{R,i}(X_{T_{D_2}})\11_{j_R(X_{T_{D_2}})=j_S(x)}\right)
    \]
    $\nu_{S,i}=\nu_{R,i}$ and
    {\color{black}\[
    \alpha_{S,n}=C\sum_{k=0}^{n}\left(\frac{\gamma}{\theta_{0,R}}\right)^k\cdot\left(\alpha_{R,n-k}+j_{0,R}\frac{k}{n}\right),
  \]}
  with the convention that $\alpha_{R,0}=1$.
\end{thm}

We emphasize that Remark~\ref{rem:slowconv} also applies to the convergence rate obtained in the last theorem.

\begin{proof}[Proof of Theorem~\ref{thm:A2}]
    As in the proof of Theorem~\ref{thm:A1}, we define 
    the linear operators $\tP:{\cal M}(W_P)\to {\cal M}(W_P)$, $\tQ:{\cal M}(W_P)\to {\cal M}(W_R)$ and $\tR:{\cal M}(W_R)\to
    {\cal M}(W_R)$ by
    \begin{align*}
        \tP \mu=\theta_0^{-1}\mu P_1,\quad \tQ \mu = \theta_0^{-1}\mu Q,\quad\text{and }\tR \mu = \theta_0^{-1}\mu R_1,
    \end{align*} 
    where  $\theta_0=\theta_{0,R}$.  Using Assumption (A2), we observe that all these operators are bounded. Our aim is to apply Proposition~\ref{prop:case2new} to $\tS:{\cal M}(W_S)\to {\cal M}(W_S)$, where ${\cal M}(W_S)\equiv {\cal M}(W_P)\oplus {\cal
    	M}(W_R)$, with $W_S=W_R+W_P$ and $\tS=\tP+\tQ+\tR$, where ${\cal M}(W_S)\equiv {\cal M}(W_P)\oplus {\cal
        M}(W_R)$, with $W_S=W_R+W_P$ and $\tS=\tP+\tQ+\tR$. Beware that $\tP,\tR,\tQ,\tS$ act on the left on $\mu$ while $P_n,R_n,Q,S_n$ act on the right, so that, for instance, $\tR \tP\mu=\theta_0^{-2}\mu P_1 R_1$.

    For all $x\in D_2$, we have $\delta_x S_n=\delta_x R_n$, so that~\eqref{eq:eta} holds true with $I_S=I_R$, $\eta_{S,i}(x)=\eta_{R,i}(x)$,  $j_S(x)=j_R(x)$ and $\alpha_{S,n}=\alpha_{R,n}$. 
    This also implies that $\theta_{S}(x)=\theta_{R}(x)$ for all $x\in D_2$.

    We fix now $x\in D_1$.
    We set 
    \[
    \tJ(x):=\max_{n\geq 0} j_R(\delta_x P_nQ)
    \]
    and consider  the operators $\tP$, $\tR$ and $\tS$ restricted to the Banach space 
    \[
    B=B_1\oplus B_2\subset \mathcal M(W_S),
    \]
    where 
    \[
    B_1=\left\{\mu\in\mathcal M(W_P),\ \max_{n\geq 0} j_R(|\mu| P_nQ)\leq \tJ(x)\right\}
    \text{ and  }B_2=\left\{\mu\in\mathcal M(W_R),\ j_R(|\mu|)\leq \tJ(x)\right\}.
    \]
    Note that $B$ is indeed stable by $\tP$, $\tR$ and $\tS$. In addition, Proposition~\ref{prop:propajs1} entails that $B_1$ is a Banach subspace of $\mathcal M(W_P)$.

    We first observe that 
    \begin{align*}
        \sum_{n=0}^\infty \VERT \tP^n\VERT\leq  \sum_{n=0}^\infty\ \theta_{0,R}^{-n}\sup_{\mu\in B_1,\ |\mu|(W_P)=1} \mu P_n W_P\leq  \frac{\theta_{0,R}}{\theta_{0,R}-\gamma}
    \end{align*}
    and $\Theta_n\leq \frac{\gamma^n/\theta_{0,R}^{n-1}}{\theta_{0,R}-\gamma}$ and $\theta_n\leq \gamma^n/\theta_{0,R}^n$ (using the notations of Proposition~\ref{prop:case2new}).

    By Lemma~\ref{lem:technicbis}, the operator $\tR:B_2\to B_2$ satisfies Assumption~(H) form Section~\ref{sec:app-2_new} with $\tJ_\tR=\tJ(x)$,
    \[
    E_\tR \mu=\sum_{i\in I_R} \mu(\11_{j_R(\cdot)=\tJ_\tR}\eta_{R,i})\,\nu_{R,i}
    \]
    and, using the fact that $\left\|\sum_{i\in I_R}\eta_{R,i}\nu_{R,i}(W_R)\right\|_{W_R}$ is finite,
    {\color{black}\[
    \alpha_{\tR,n}=C\left(\alpha_{R,n}+\frac{\11_{\tJ_\tR\geq 1}}{n}\right)
    \]
    for some constant $C>0$.}
    We thus deduce from Proposition~\ref{prop:case2new} that $\tS$ restricted to $B$ satisfies Assumption~(H) with $\tJ_\tS=\tJ(x)$, for all
    $\mu\in B$,
    \begin{align}
        E_\tS\mu&=E_\tR\mu+\sum_{k=0}^\infty E_\tR\tQ\tP^k\mu \notag\\
        &=\sum_{i\in I_R} \left(\mu(\11_{j_R(\cdot)=\tJ_\tR}\eta_{R,i})+\sum_{k=0}^\infty  \theta_{0,R}^{-k-1} \mu P^k
        Q(\11_{j_R(\cdot)=\tJ_\tR}\eta_{R,i})\right)\,\nu_{R,i} \notag\\
        &
        =\sum_{i\in I_R}\EE_\mu\left(\theta_{0,R}^{- T_{D_2}}\eta_{R,i}(X_{T_{D_2}})\11_{j_R(X_{T_{D_2}})=\tJ(x)}\right)\,\nu_{R,i}\label{eq:preuve-A2new}
    \end{align}
    and there exists a constant $C$ independent of $x\in D_1$ such that 
    \begin{align}
        \alpha_{\tS,n}&=\alpha_{{\tR},n}+ C \sum_{k=0}^{n-1} \alpha_{{\tR},n-k-1}\theta_{k}+ C\Theta_n+C\,\sum_{k=0}^{n-1} \frac{\tJ_{\tS}\,k}{n}\,\theta_k\nonumber\\
        &\leq C\sum_{k=0}^{n}\left(\alpha_{R,n-k}+\frac{\11_{\tJ(x)\geq 1}}{n-k+1}+\tJ(x)\frac{k}{n}\right)\left(\frac{\gamma}{\theta_{0,R}}\right)^k\label{eq:xdependentstep2}
        \\
        &\leq \alpha_{S,n}:=C\sum_{k=0}^{n}\left(\alpha_{R,n-k}+j_{0,R}\frac{k}{n}\right)\left(\frac{\gamma}{\theta_{0,R}}\right)^k,\nonumber
    \end{align}
    where $\alpha_{R,0}:=1$.  Since $\tS^n=\theta_{0,R}^{-n} S_n$, taking $\mu=\delta_x$ in~\eqref{eq:preuve-A2new}, we finally deduce that, for all $x\in D$ and all $f\in L^\infty(W_S)$,
    \begin{multline}
        \label{eq:finstep2new}
        \left|\theta_{0,R}^{-n}n^{-\tJ(x)}S_n f(x)- \sum_{i\in I_R} 
        \EE_x\left(\theta_{0,R}^{- T_{D_2}}\eta_{R,i}(X_{T_{D_2}})\11_{j_R(X_{T_{D_2}})=\tJ(x)}\right)
        \,\nu_{R,i}(f)\right|\\
        \leq \alpha_{S,n}W_S(x)\|f\|_{W_S},
    \end{multline}
    where we extended $\tJ$ to $D_2$ by setting $\tJ(x):=j_R(x)$ if $x\in D_2$. 
    
    In order to conclude, it remains to prove that $\theta_{0,S}=\theta_{0,R}$ and that $j_S(x)=\tJ(x)$ for all $x\in D$. Inequality~\eqref{eq:finstep2new} with $f=\11_D$ implies 
    that $\theta_{S}(x)\leq \theta_{0,R}$ for all $x\in D$, so that $\theta_{0,S}\leq \theta_{0,R}$. Moreover, for all $x\in D_2$,
    $\theta_{S}(x)=\theta_{R}(x)$, and thus $\theta_{0,S}\geq \theta_{0,R}$. We deduce that $\theta_{0,S}=\theta_{0,R}$ and hence, using
    again~\eqref{eq:finstep2new} with $f=\11_D$, we deduce that $j_S(x)\leq \tJ(x)$ for all $x\in D$. On the one hand, for all $x\in D_2$,
    we have $j_S(x)=j_R(x)=\tJ(x)$. On the other hand, for $x\in D_1$, we observe that, for any $n\geq 0$ such that $\tJ(x)=j_S(\delta_x
    P_n Q)$, we have the inequality $\delta_x S^{n+1}\11_D\geq \delta_x P_n Q\11_D$, and hence
    \[
    j_S(x)=j_S(\delta_x S^{n+1})\geq j_S(\delta_x P_n Q)=\tJ(x).
    \]
    We thus proved that $j_S(x)\geq \tJ(x)$ for all $x\in D$, which concludes the proof of Theorem~\ref{thm:A2}.
\end{proof}

\subsection{\color{black} Case where $D_1$ is a critical sink ($\theta_{0,P}=\theta_{0,R}$)} 

\label{sec:D1criticalsink}

{\color{black} In this section, we consider the situation where the process evades $D_1$ and $D_2$  at the same pace. This is made precise in the following assumption, which will allow us to make use of Proposition~\ref{prop:case3new}.}

\medskip \noindent \textbf{Assumption (A3)} We have $j_{0,P}=0$, $j_{0,R}<+\infty$ and $\theta_{0,R}=\theta_{0,P}$. In addition, the process $X$ restricted to $D_1$ satisfies Assumption~(A) with  $\eta_{P}>0$, 
 and the process $X$ restricted to $D_2$ also satisfies Assumption~(A). Finally,
 \begin{equation}
   \label{eq:hyp-A3}
   \E_x(W_R(X_1))\leq W_P(x),\ \forall x\in D_1,   
 \end{equation}
and there exists $\ell_*\in\ZZ_+$ such that, for all $x\in D_1$ and all $i\in I_P$,
\begin{gather}
\label{eq:definelstar}
\P_x\left(j_R(X_{T_{D_2}})\leq \ell_*\text{ and }T_{D_2}<+\infty\right)=\P_x(T_{D_2}<+\infty)\\
\text{and}\quad \P_{\nu_{P,i}}\left(j_R(X_{T_{D_2}})=\ell_*
\text{ and }
 \eta_R(X_{T_{D_2}})>0 \text{ and }T_{D_2}<+\infty\right)>0, 
\label{eq:definelstar-bis}
\end{gather}
where we recall that $\eta_R=\sum_{k\in I_R}\eta_{R,k}$.
Note that $\ell_*\leq j_{0,R}$.

\medskip

\begin{rem}
	In the above Assumption~(A3), the assumptions $j_{0,P}=0$ and the fact that~\eqref{eq:definelstar} is satisfied for all $x\in D_1$ may seem restrictive
	conditions. However, we will see in Section~\ref{sec:finite-classes} that, applying this property inductively in a precise
        order, this is sufficient to obtain Condition~(A) with non-zero $j_{0,S}$ in cases with a finite or denumerable number of communication classes.
	\erem
\end{rem}

\medskip

{\color{black} We emphasize that Remark~\ref{rem:A1withconstants} also applies to Assumption~(A3).  The following theorem states that Assumption~(A2) implies Assumption~(A), with explicit parameters. In this situation the limiting distribution of the process starting from $D_1$ only charges $D_2$.}

\medskip

\begin{thm}
	\label{thm:A3}
	Assume that Assumption~(A3) holds true. Then $X$ satisfies Assumption~(A) with $W_S=W_P+W_R$.
        Moreover, there exists a constant $C>0$, independent of $x\in D$ and $n\in\NN$, such that
	 we have $\theta_{0,S}=\theta_{0,R}=\theta_{0,P}$, 
		\[
		j_S(x)=\begin{cases}
   		1+\ell_*&\text{ for all }x\in D_1,\\
		j_R(x)&\text{ for all }x\in D_2,
		\end{cases}
		\]
where $\ell_*$ is defined in~\eqref{eq:definelstar}--\eqref{eq:definelstar-bis},  $I_S=I_R$,
		for all $i\in I_R$, $\nu_{S,i}=\nu_{R,i}$,
\[
\eta_{S,i}(x)=\eta_{R,i}(x)+\frac{\theta_{0,P}^{-1}}{1+\ell_*}\sum_{k\in I_P}\eta_{P,k}(x)
\EE_{\nu_{P,k}}\left(\eta_{R,i}(X_1)\11_{j_R(X_1)=\ell_*}\right),\quad\forall x\in D,
\]
and
		\[
		\alpha_{S,n}=\frac{C}{n}\left(\ell_*+\sum_{k=0}^n \left(\alpha_{P,k}+\alpha_{R,k}\frac{k^{\ell_*}}{n^{\ell_*}}\right)\right),
              \]
              with the convention that $\alpha_{P,0}=\alpha_{R,0}=1$.
\end{thm}

\begin{rem}
	In the conclusion of the last theorem, even if $\alpha_{P,n}$ and $\alpha_{S,n}$ converge geometrically to $0$,
	$\alpha_{S,n}$ only converges to $0$ in $O(1/n)$. 
	\erem
\end{rem}

\begin{proof}[Proof of Theorem~\ref{thm:A3}]
	As in the proof of the two previous results, we define 
	 the linear operators $\tP:{\cal M}(W_P)\to {\cal M}(W_P)$, $\tQ:{\cal M}(W_P)\to {\cal M}(W_R)$ and $\tR:{\cal M}(W_R)\to
         {\cal M}(W_R)$  by
	\begin{align*}
	\tP \mu=\theta_0^{-1}\mu P_1,\quad \tQ \mu = \theta_0^{-1}\mu Q,\quad\text{and }\tR \mu = \theta_0^{-1}\mu R_1,
	\end{align*} 
	where  $\theta_0=\theta_{0,P}=\theta_{0,R}$.
Assumption (A3) entails that all these operators are bounded. Our aim is to apply Proposition~\ref{prop:case3new} to $\tS:{\cal M}(W_S)\to {\cal M}(W_S)$, where ${\cal M}(W_S)\equiv {\cal M}(W_P)\oplus {\cal
          M}(W_R)$, with $W_S=W_R+W_P$ and $\tS=\tP+\tQ+\tR$. Beware that $\tP,\tR,\tQ,\tS$ act on the left on $\mu$ while $P_n,R_n,Q,S_n$ act on the right, so that, for instance, $\tR \tP\mu=\theta_0^{-2}\mu P_1 R_1$.

 If $x\in D_2$, then~\eqref{eq:eta} holds true with $I_S=I_R$ and $\eta_{S,i}(x)=\eta_{R,i}(x)$,  $j_S(x)=j_R(x)$ and $\alpha_{S,n}=\alpha_{R,n}$.

 Fix $x\in D_1$. We consider $\tP$, $\tR$ and $\tS$ restricted to the Banach space
    \[
 B=B_1\oplus B_2\subset \mathcal M(W_S),
 \]
 where 
 \[
 B_1=\mathcal M(W_P)
 \text{ and  }B_2=\left\{\mu\in\mathcal M(W_R),\ j_R(|\mu|)\leq \ell_*\right\}
 \]
 Note that it follows {\color{black}from} Proposition~\ref{prop:propajs1} and the assumption that 
 {\color{black}\[
 \PP_X(j_R(X_{T_{D_2}})\leq \ell_*\text{ and }T_{D_2}<\infty)=\PP_X(T_{D_2}<\infty),
 \]} that $\tQ
 B_1\subset B_2$, and from the rest of Assumption~(A3) that $\tP:B_1\to B_1$,
 $\tR:B_2\to B_2$ and $\tQ:B_1\to B_2$ are bounded operators.

   As in the previous step, the operator $\tR$ satisfies Assumption~(H) in the Appendix with $\tJ_\tR=\ell_*$, 
   \[
   E_\tR \mu=\sum_{i\in I_R} \mu(\11_{j_R(\cdot)=\ell_*}\eta_{R,i})\,\nu_{R,i},
   \] 
   and
  \[
  \alpha_{\tR,n}=C\left(\alpha_{R,n}+\frac{\11_{\ell_*\geq 1}}{n}\right),
  \]
  for some constant $C>0$.
Moreover, Assumption~(A) for $P_n$ implies that $\tP$ satisfies Assumption~(H) from Section~\ref{sec:app-2_new} with $\tJ_\tP=0$, $E_\tP\mu=\sum_{j\in I_P}\mu(\eta_{P,j})\nu_{P,j}$ and $\alpha_{\tP,n}=\alpha_{P,n}$.
We conclude from Proposition~\ref{prop:case3new} that $\tS$ satisfies Assumption~(H) with $\tJ_\tS=1+\ell_*$,
\[
E_\tS\mu=\frac{1}{\tJ_{\tS}}E_{\tR}{\tQ}E_{\tP}\mu=\frac{\theta_{0,P}^{-1}}{1+\ell_*}\sum_{i\in I_R}\sum_{k\in I_P}\mu(\eta_{P,k}) \EE_{\nu_{P,k}}\left(\11_{j_R(X_1)=\ell_*}\eta_{R,i}(X_1)\right)\nu_{R,i},
\]
and
\begin{align*}
\alpha_{\tS,n}&=\frac{C}{n}\left( 
\tJ_{\tR}+\sum_{k=0}^n \alpha_{{\tP},k}+  \left(\max_{k\geq 0}\alpha_{{\tP},k}+1\right)\sum_{k=0}^{n} \alpha_{{\tR},n-k} \left(\frac{n-k}{n}\right)^{\tJ_\tR}\right)\\
& 
\leq \alpha_{S,n}:=\frac{C}{n}\left(\ell_*+\sum_{k=0}^n \left(\alpha_{P,k}+\alpha_{R,k}\frac{k^{\ell_*}}{n^{\ell_*}}\right)\right),
\end{align*}
with $\alpha_{P,0}=\alpha_{R,0}=1$.
Since $\tS^n\mu=\theta_{0,P}^{-n}\mu S_n$, we deduce that, for all $x\in D_1$ and all $f\in L^\infty(W_S)$,
\begin{multline}
\label{eq:finstep3}
\left|\theta_{0,P}^{-n} n^{-(1+\ell^*)} S_n f(x)-
\frac{\theta_{0,P}^{-1}}{1+\ell_*}\sum_{i\in I_R}\sum_{k\in I_P}\eta_{P,k}(x) \EE_{\nu_{P,k}}\left(\11_{j_R(X_1)=\ell_*}\eta_{R,i}(X_1)\right)\nu_{R,i}(f)
\right|\\
\leq \alpha_{S,n}W_S(x)\|f\|_{W_S}.
\end{multline}
This implies that $\theta_{S}(x)\leq \theta_{0,P}$ for all $x\in D_1$, so $\theta_{0,S}\leq
\theta_{0,P}\vee\theta_{0,R}=\theta_{0,P}=\theta_{0,R}$. Conversely, since $S f\geq R f$ for all positive $f$, we have
$\theta_{0,S}\geq \theta_{0,R}$. We thus deduce that $\theta_{0,S}=\theta_{0,P}=\theta_{0,R}$. We also have, by definition of $S$,
$j_S(x)=j_R(x)$ for all $x\in D_2$. Moreover, for all $x\in D_1$,~\eqref{eq:finstep3} implies that $j_S(x)\leq 1+\ell^*$.

It remains to prove that $j_S(x)\geq 1+\ell_*$ for all $x\in D_1$. 
    Fix $x\in D_1$ until the end of the proof. Since $\eta_{P}(x)>0$, we
  deduce from Proposition~\ref{prop:mainbis}~(i) that
\begin{align}
    \label{eq:defnu}
\nu:=\frac{1}{\eta_P(x)}\sum_{k\in I_P}\eta_{P,k}(x)\nu_{P,k}
\end{align}
is a quasi-stationary distribution for the semigroup $(P_n)_{n\geq 0}$ with exponential convergence parameter $\theta_{0,P}$. Let us
first prove that $j_R(\nu Q)=\ell_*$. Since $(R_n)_{n\geq 0}$ satisfies (A), we have for all $y\in D_2$ such that $j_R(y)\leq\ell_*$
\[
\theta_{0,R}^{-n}n^{-\ell^*}\delta_y R_n\11_{D_2}\xrightarrow[n\rightarrow+\infty]{}
\begin{cases}
  \sum_{i\in I_R}\eta_{R,i}(y) & \text{if $j_R(y)=\ell^*$,} \\ 0 & \text{if $j_R(y)<\ell_*$,}
\end{cases}
\]
where the convergence holds in $L^\infty(W_R)$. Therefore,
\begin{equation}
  \label{eq:youpi}
\theta_{0,R}^{-n}n^{-\ell^*}\nu Q R_n\11_{D_2}\xrightarrow[n\rightarrow+\infty]{}\sum_{i\in I_R}\nu Q\left(\eta_{R,i}\11_{j_R(\cdot)=\ell_*}\right).  
\end{equation}
Now, using that $\nu$ is a quasi-stationary distribution,
\begin{align*}
  \E_\nu\left(\eta_{R}(X_{T_{D_2}}) \11_{j_R(X_{T_{D_2}})=\ell_*}\11_{T_{D_2}<\infty}\right) & =\sum_{m=0}^{+\infty}\nu P_m Q \left(\eta_R(\cdot)\11_{j_R(\cdot)=\ell_*}\right)
  \\ & =\frac{1}{1-\theta_{0,P}}\nu Q \left(\eta_R(\cdot)\11_{j_R(\cdot)=\ell_*}\right).
\end{align*}
{\color{black}By~\eqref{eq:definelstar-bis} in Assumption~(A3) and given the definition of~$\nu$ in~\eqref{eq:defnu},} the left-hand side is positive, so we have proved that
$\theta_{0,R}^{-n}n^{-\ell^*}\nu Q R_n\11_{D_2}$ converges to a positive limit. This shows that $j_R(\nu Q)=\ell_*$.

For all $n\geq 1$, using the fact that $S_n=P_n+R_n+\sum_{k=0}^{n-1} P_{n-k-1}QR_k$, 
we have
\begin{align*}
n^{-(\ell_*+1)}\theta_0^{-n}\nu S_n\11_D&\geq n^{-(\ell_*+1)}\theta_0^{-n}\sum_{k=0}^{n-1} \nu P_{n-k-1}QR_{k}\11_{D_2}\\
&=n^{-(\ell_*+1)}\theta_0^{-1}\sum_{k=1}^{n-1} \theta_0^{-k}\nu QR_{k}\11_{D_2}\\
&=n^{-(\ell_*+1)}\theta_0^{-1}\sum_{k=1}^{n-1} k^{\ell_*}\left[k^{-\ell_*}\theta_0^{-k}\nu QR_{k}\11_{D_2}\right].
\end{align*}
Using that, by~\eqref{eq:youpi},
$k^{-\ell_*}\theta_0^{-k}\nu QR_{k}\11_{D_2}$ converges to a positive limit when $k\to+\infty$, we deduce that
\[
\liminf_{n\rightarrow+\infty}n^{-(\ell_*+1)}\theta_0^{-n}\nu S_n\11_D>0.
\]
This shows that $j_S(\nu)\geq \ell_*+1$. {\color{black}In addition,} we have $\theta_0^{-n}\delta_x P_n\to \sum_{k\in I_P}\eta_{P,k}(x)\nu_{P,k}=\eta_P(x)\nu$ in
$\mathcal M(W_S)$ when $n \to+\infty$. Since $\eta_{P}(x)>0$, we deduce from the lower semi-continuity of $j_S$ (see
Proposition~\ref{prop:propajs1}) that
\[
\liminf_{n\to+\infty} j_S(\theta_0^{-n}\delta_x P_n)\geq j_S(\eta_P(x)\nu)= j_S(\nu)=\ell_*+1.
\]
Using again Proposition~\ref{prop:propajs1}, we have $j_S(x)=j_S(\theta_0^{-n}\delta_x S_n)\geq j_S(\theta_0^{-n}\delta_x P_n)$ for
all $n\geq 0$. So we finally deduce that
\[
j_S(x)=\liminf_{n\to+\infty} j_S(\theta_0^{-n}\delta_x S_n)\geq \liminf_{n\to+\infty} j_S(\theta_0^{-n}\delta_x P_n)\geq j_S(\nu_P)=\ell_*+1.
\]
This concludes the proof of Theorem~\ref{thm:A3}.
\end{proof}

\section{Reducible state spaces with several communication classes}
\label{sec:finite-classes}


Our goal is to study  quasi-stationary distributions on general reducible state spaces, a situation which naturally leads to non-zero polynomial convergence parameters. In particular, we extend the results of~\cite[Section~6.2]{ChampagnatVillemonais2017b}, which are stated under conditions ensuring that the polynomial convergence parameter of the process vanishes.
We refer the reader
to~\cite{DoornPollett2008,DoornPollett2009}  where an in-depth study of the quasi-stationary distributions on finite reducible state spaces has been conducted (see also the survey~\cite{DoornPollett2013}, an earlier work~\cite{Mandl1959} summarized in \cite[Section~9]{darroch-seneta-65}, and the more recent works~\cite{ChampagnatDiaconisEtAl2012,BenaimCloezEtAl2016}). The quasi-stationary distribution of particular processes on reducible state spaces with finitely many communication classes have also previously been studied in \cite{Ogura1975} (for multi-type Galton-Watson processes), \cite[Section~3]{Gosselin2001} (for discrete state space processes, under conditions ensuring that the polynomial convergence parameter vanishes),  and~\cite{ChampagnatRoelly2008} (for multitype
Dawson-Watanabe processes).

We consider a Markov process $X$ with semigroup $S$ on a general state space $D$ that can be decomposed into finitely many disjoint sets $E_\emptyset,E_1,E_2,\ldots,E_k$, where $k\geq 1$. We denote, for
all $i\in\{\emptyset,1,\ldots,k\}$, by $Y^{(i)}$ the process
\[
Y^{(i)}_n=
\begin{cases}
  X_n & \text{if }n<T_{\cup_{j\neq i}E_j\,\cup\,\{\d\}}, \\
  \d & \text{otherwise}
\end{cases}
\]
and define by $\theta_{0,i}$ its exponential convergence parameter.  The process $Y^{(i)}$ is called the process $X$ restricted to $E_i$. More generally, for all $M\subset D$, we call process $X$
restricted to $M$ the process $X$ killed after its first exit time from $M$. 

{\color{black}
  We introduce a set of assumptions ensuring that the classes $E_1,\ldots,E_k$ all have the same exponential convergence parameter and
that class $E_\emptyset$ has a smaller exponential convergence parameter but satisfy less stringent assumptions.


\textbf{Assumption (B1).} We  assume that, for all $i\in\{1,\ldots,k\}$,
the process $Y^{(i)}$ satisfies Assumption~(A) with the objects  $\theta_{0,i}$,
 $j_i$, $\alpha_{i,n}$, $W_i$,
$I_i=\{1\}$, $\nu_i$ and $\eta_i$ (note that we omit the second index for $\eta_{i,1}$
and $\nu_{i,1}$). We also assume that $j_i\equiv 0$, $\eta_{i}>0$ on $E_i$
and
 \begin{align}
 \label{eq:ass1}
 \theta_{0,i}=\bar \theta
 \end{align}
for some constant $\bar\theta$ independent of $i\in\{1,\ldots,k\}$.
}

\medskip
We emphasize that many references provide practical criteria to check Assumption (A) with $I_S=\{1\}$, $j_S\equiv 0$, $\eta_S>0$
and with $\alpha_{S,n}$ converging exponentially fast to $0$, which corresponds to the classical irreducible situation,
{\color{black} see~\cite{GongQianEtAl1988,GuillinNectouxEtAl2021,LelievreRamilEtAl2021,BenaimChampagnatEtAl2021} for diffusion processes,
~\cite{ChampagnatVillemonais2016b,ChampagnatVillemonais2017b,ChampagnatVillemonais2019,BansayeCloezEtAl2019} for general criteria
based on semi-group arguments,  \cite{FerreRoussetStoltz2018,HinrichKolbEtAl2018,GuillinNectouxEtAl2020} for general criteria based
on regularity properties of the semigroup.}

{\color{black}
The following assumption ensures that the sets $E_1,\ldots,E_k$ behave like communication classes.

\textbf{Assumption (B2).} We assume that the set $\{1,\ldots,k\}$ can be equipped with a partial strict order~$\prec$ such that $i\prec j$ if and only if
$E_i$ is accessible from $E_j$ in the sense that: for all $i,j\in \{1,\ldots,k\}$,
if $i\prec j$, then
\begin{equation}
  \label{eq:accessible-reducible}
\forall x\in E_j,\quad \PP_x(T_{E_i}<+\infty)>0
\end{equation}
and, if $i\not\prec j$, then
\[
\forall x\in E_j,\quad \mathbb{P}_x(T_{E_i}<+\infty)=0.
\]
and
 \[
 \forall x\in E_j,\quad \mathbb P_x(\exists n\geq T_{E_j^c}\text{ such that }X_n\in E_j)=0,
 \]   
 where $E_j^c=D\cup\{\partial\}\setminus E_j$.

\medskip

Our next assumption states that the exit time from the set $E_\emptyset$ (which may not, in general, satisfy the properties of other classes
given in Assumptions (B1) and (B2)) is smaller than the exit time from the sets $E_1,\ldots,E_k$.

\textbf{Assumption (B3).}
We assume that there exists $\gamma<\bar \theta$  and a function $W_\emptyset\geq 1$  such that, for all $x\in E_\emptyset$ and for some constant $C>0$,
\begin{align}\label{as:weakweakA}
	\E_x(W_\emptyset(Y^{(\emptyset)}_n))\leq C\gamma^n W_\emptyset(x),\ \forall n\geq 0.
\end{align}

\medskip

Our last assumption gives a consistency property between the functions $W_i$ when the process jumps from a class $E_i$ to another.

\textbf{Assumption (B4).}
We set $W=W_\emptyset+\sum_{i=1}^k W_i$ and assume that there exists a constant $C_W>0$ such that
\begin{align}
	\label{eq:appli1W}
	\E_x(W(X_1))\leq C_W W(x), \quad\forall x\in D.
\end{align}
}

To state the main result of this section, we introduce the following notations. We define the set
\[
F_0:=\left\{\text{minimal elements in }\{1,\ldots,k\}\text{ for the partial order }\prec\right\}.
\]
  and, for all $\ell\geq 0$, we define by induction
\[
F_{\ell+1}:=\left\{\text{minimal elements in }\{1,\ldots,k\}\setminus\left(F_0\cup\ldots\cup F_\ell\right)\text{ for the order }\prec\right\}.
\]
We denote by $\bar\ell$ the non-negative integer such that $F_\ell=\emptyset$ iff $\ell>\bar{\ell}$.
For all $x\in \bigcup_{i=1}^k E_i$, we define
$\text{index}(x)$ as the unique $\ell\in\{0,\ldots,\bar\ell\}$ such that $x\in E_i$ for some $i\in F_\ell$.
For all $x\in E_\emptyset$, we also define
\[
\text{index}(x)=\max\{\ell\geq 0\text{ such that }\exists i\in F_{\ell},\ \mathbb P_x(T_i<\infty)>0\},
\]
with $\max \emptyset = -1$.


\begin{thm}
  \label{thm:finite-classes}
  {\color{black}Under Assumptions (B1), (B2), (B3) and (B4),} the process $X$ satisfies Condition~(A) with 
  \begin{gather*}
    \theta_{0,S}=\bar{\theta},\quad I_S= F_0,\quad W_S=W_\emptyset+\sum_{i=1}^k W_i, \\
    j_S(x)=\text{\emph{index}}(x)\vee 0,\text{ for all }x\in D,
  \end{gather*}
  and, for all $i\in I_S$,
  \begin{gather}
  \label{eq:thmfinite1}
  \nu_{S,i}\propto \nu_i+\sum_{\ell\geq 0}\theta_{0,S}^{-\ell-1}\P_{\nu_i}(X_{1}\not\in E_i,\ X_{\ell+1}\in\cdot), \\
    \label{eq:thmfinite2}
  \eta_{S,i}(x)>0\text{ for all }x\in D\text{ with }\mathbb P_x(T_{E_i}<\infty)>0
  \end{gather}
  and
  \begin{align}
    \label{eq:thmfinite3}
  \eta_{S,i}(x)=0\text{ for all }x\in D\text{ with }\mathbb P_x(T_{E_i}<\infty)=0.
  \end{align}
\end{thm}

\begin{rem}
  In this theorem, the functions $\eta_{S,i}$ can also be expressed in terms of the parameters of the problem, since they are constructed in the proof below
  with an inductive argument, with explicit expressions at each step.
  \erem
\end{rem}

\begin{rem}
  \label{rem:1/n}
  Similarly, the speed of convergence $\alpha_{S,n}$ is also constructed expli\-ci\-te\-ly with an inductive argument in the proof below.
  In particular, if it is assumed that $\alpha_{i,n}$ converges exponentially fast to 0 for all $i$ such that
  $\theta_{0,i}=\bar\theta$, one can easily check that $\alpha_{S,n}$ also converges to 0 exponentially fast if $j_S\equiv 0$, and
  converges to 0 polynomially in $O(1/n)$  otherwise.
  \erem
\end{rem}

\begin{rem}
\label{rem:recursif-QSD}
It follows from the last theorem and Corollary~\ref{cor:corQSD} that the set of quasi-stationary distributions $\nu$ for $X$ such
that $\nu(W_S)<+\infty$ and $\nu(\eta_S)>0$ has dimension $\# {F}_0$ and is spanned (in the sense of convex hulls) by $\nu_{S,i}$,
$i\in {F}_0$. Our result also allows to characterize all quasi-stationary distributions $\nu$ of $X$ such that $\nu(W_S)<\infty$: one
can obtain the other quasi-stationary distributions by applying Theorem~\ref{thm:finite-classes} (assuming its assumptions are
satisfied) to the process $X$ restricted to the subset of $E_\emptyset$ composed of points from which $E_1\cup\ldots\cup E_k$ is not
accessible, i.e.\ $\{x\in E_\emptyset,\ \text{index}(x)=-1\}$, and by proceeding recursively. All the new quasi-stationnary
distributions obtained this way have an exponential convergence parameter (strictly) smaller than $\bar{\theta}$. {\color{black} This way of
enumerating quasi-stationary distributions is related to the enumeration of equilibria in the epidemic model
of~\cite{delmas-et-al-2025} using what they call ``supercritical antichains''}.

{\color{black} In the particular case where the state space $D$ is finite, our result are thus reminiscent of~\cite{DoornPollett2009} (see in particular Theorems~4.3 and~5.1 therein). These results are already quite complete, and one of our main contributions to the problem in the finite state space situation is to determine explicitely the polynomial convergence parameter associated to each communication class, and to emphasize the support of the functions $\eta_{S,i}$.}
\erem
\end{rem}


\begin{rem}
  {\color{black} The above result also allows to study reducible processes with denumerably many communication classes. In particular,
    $E_\emptyset$ may contain infinitely many communication classes. In particular, our proof applies to cases where there exists a
    denumerable sequence $(E_i)_i\geq 1$ satisfying (B1), (B2) and (B4) such that $F_\ell$ is a finite set for all $\ell\geq 0$.}
  \erem
\end{rem}


\begin{proof}[Proof of Theorem~\ref{thm:finite-classes}]
    
  In what follows, we set, for all index value $\ell\in\{-1,0,1,,\ldots,\bar\ell\}$,
  \begin{align*}
      E_\emptyset^{(\ell)}:=\{x\in E_\emptyset,\text{ such that index}(x)=\ell\}.
  \end{align*}
  The proof is based on an induction argument, based on a specific decomposition of the state space $D\cup\{\d\}$ into an increasing
  sequence of closed subsets, as defined below. We call a subset $\bar{D}$ of $D\cup\{\d\}$ a closed set if for all $x\in D$,
  $\mathbb{P}_x(\exists n\geq 0,\ X_n\not\in \bar D)=0$. We first observe that $E_\emptyset^{(-1)}\cup\{\d\}$ is closed and, by Assumption (B2)
  and the definition of $F_0$,
  \begin{equation}
    \label{eq:set-0}
    S_0:=\bigcup_{i\in F_0} E_i\cup E^{(-1)}_\emptyset\cup\{\d\}
  \end{equation}
  is also closed. Similarly, the sets
  \begin{equation}
    \label{eq:set-1}
    S_n:=\bigcup_{\ell=0}^{n}\bigcup_{i\in F_\ell} E_i \cup \bigcup_{k=-1}^{n-1} E_\emptyset^{(k)}\cup\{\d\}
  \end{equation}
  for all $n\in\{1,\ldots,\bar\ell\}$, and
  \begin{equation}
    \label{eq:set-2}
    S'_n:=\bigcup_{\ell=0}^{n}\bigcup_{i\in F_\ell} E_i \cup \bigcup_{k=-1}^{n} E_\emptyset^{(k)}\cup\{\d\}
  \end{equation}
  for all $n\in\{0,\ldots,\bar\ell\}$, are also closed. Below, we prove by induction that the following property is true on any of
  the previous sets. Given a closed subset $\bar{D}$ of $D\cup\{\d\}$, we say that property (P) is satisfied on $\bar D$ if the
  semi-group $R$ of the process $X$ restricted to $\bar D$ satisfies Assumption (A) with $\theta_{0,R}=\bar\theta$, $I_{R}=F_0$,
  $W_R$ the restriction of $W_S$ to $\bar D\setminus\{\d\}$, $j_R(x)=\text{index}(x)\vee 0$ for all $x\in \bar D\setminus\{\d\}$, for
  all $i\in I_R$,
  \begin{gather*}
    \nu_{S,i}\propto \nu_i+\sum_{\ell\geq 0}\theta_{0,S}^{-\ell-1}\P_{\nu_i}(X_{1}\not\in E_i,\ X_{\ell+1}\in\cdot), \\
    \eta_{S,i}(x)>0\text{ for all }x\in \bar D\text{ with }\mathbb P_x(T_{E_i}<\infty)>0
  \end{gather*}
  and
  \[
  \eta_{S,i}(x)=0\text{ for all }x\in \bar D\text{ with }\mathbb P_x(T_{E_i}<\infty)=0.
  \]
  This will prove Theorem~\ref{thm:finite-classes} since, by definition of $\bar\ell$, $D=S'_{\bar\ell}$.
  
  \medskip\textbf{Step 1. Proof that (P) is satisfied on the set $S_0$.}
  Our aim is to apply Theorem~\ref{thm:A1}  with 
  \[
  D_1=\cup_{i\in F_0} E_i\text{ and }D_2=E_\emptyset^{(-1)}.
  \]
  In what follows, we set $W_R=W\11_{D_2}$ and $W_P=C_W W\11_{D_1}$. According to~\eqref{eq:appli1W} in Assumption~(B1), we have, for all $x\in D_1$,
  \[
  \mathbb E_x(W_R(X_1))\leq\mathbb E_x(W(X_1))\leq C_W W(x)=W_P(x),
  \]
  so that the first part of~\eqref{eq:A1} holds true. In addition, since $D_2$ is closed, we have, for all $x\in D_2$, 
  \begin{align*}
      \mathbb E_x(W_R(X_n))=\mathbb E_x(W_\emptyset(Y^{(\emptyset)}_n))\leq C\gamma^n W_\emptyset(x)=C\gamma^n W_R(x),
  \end{align*}
where we used~\eqref{as:weakweakA} from Assumption~(B3) for the last inequality.  Hence the second part  of~\eqref{eq:A1} holds true.

Finally, for all $i\in F_0$ and $x\in E_i$, we have by Assumption~(B1)
\begin{align*}
      \left|\bar\theta^{-n} \EE_x(f(X_n)\11_{n< T_{\cup_{j\neq i}E_j\cup\{\partial\}}})- \eta_{i}(x)\nu_{i}(f)\right|\leq \alpha_{i,n} W_i(x)\|f\|_{W}.
\end{align*}
Since $D_1\cup D_2\cup\{\partial\}$ is closed, we deduce that this reduces to
\begin{align*}
    \left|\bar\theta^{-n} \EE_x(f(X_n)\11_{n< T_{\cal A}})- \eta_{i}(x)\nu_{i}(f)\right|\leq \alpha_{i,n} W_i(x)\|f\|_{W},
\end{align*}
where
\[
  {\cal A}:={\cup_{j\in  F_0,j\neq i}E_j\cup D_2\cup \{\partial\}}.
\]
Since in addition, by definition of $F_0$,  $j\not\prec i$ for all $i\neq j\in  F_0$, we deduce that
\begin{align*}
    \left|\bar\theta^{-n} \EE_x(f(X_n)\11_{n< T_{D_2\cup \{\partial\}}})- \eta_{i}(x)\nu_{i}(f)\right|\leq \alpha_{i,n} W_i(x)\|f\|_{W}.
\end{align*}
Summing over $i\in F_0$, we conclude that, for all $x\in D_1$,
\begin{align*}
    \left|\bar\theta^{-n} \EE_x(f(X_n)\11_{n< T_{D_2\cup \{\partial\}}})- \sum_{i\in F_0} \eta_{i}(x)\nu_{i}(f)\right|\leq {\color{black}\sum_{i\in F_0} \alpha_{i,n} W_i(x)\|f\|_{W}.}
\end{align*}
In particular, the process restricted to $D_1$ satisfies Assumption~(A) with $j_{0,S}=0$. We conclude that Assumption~(A1) holds true and hence, according to Theorem~\ref{thm:A1}, that the process restricted to $D$ satisfies Assumption~(A) with $W_S=W\11_{D_2}+C_W W\11_{D_1}$, $j_S\equiv 0$, $I_S=F_0$,
\begin{align*}
    \nu_{S,i}\propto \nu_i+\sum_{\ell\geq 0}\theta_{0,S}^{-\ell-1}\P_{\nu_i}(X_{1}\not\in E_i,\ X_{\ell+1}\in\cdot),
\end{align*}
and, for all $i\in I_S$, $\eta_{S,i}(x)>0$ if and only if $x\in E_i$ with $i\in F_0$. This proves that property (P) is satisfied on
the set $S_0$.

  \medskip\textbf{Step 2. Proof that (P) is satisfied on a set of the form~\eqref{eq:set-2} assuming it is satisfied on a set of the
    form~\eqref{eq:set-1}.} Assume that Property (P) is satisfied on the set $S_n$ for some $n\in\{0,1,\ldots,\bar\ell\}$. Let us
  prove that it is satisfied on the set $S'_n$.
  To do so, we aim to apply Theorem~\ref{thm:A2} with
  \begin{align*}
      D_1=E_\emptyset^{(n)}\text{ and }D_2=\bigcup_{\ell=0}^{n}\bigcup_{i\in F_\ell} E_i \cup \bigcup_{\ell=-1}^{n-1} E_\emptyset^{(\ell)}.
  \end{align*}
  Our induction assumption applies to $D_2\cup\{\d\}$, so that the semi-group $R$ of
  the process $X$ restricted to $D_2\cup\{\d\}$ satisfies Assumption~(A) with 
  \begin{gather*}
      \theta_{0,R}=\bar{\theta},\quad I_R={F}_0,\quad W_R={\color{black}W_\emptyset+}\sum_{\ell=0}^{n}\sum_{i\in F_\ell}^k W_i, \\
      j_R(x)=\text{index}(x)\vee 0\text{ for all }x\in \bigcup_{\ell=0}^{n}\bigcup_{i\in F_\ell} E_i\cup \bigcup_{\ell=-1}^{n-1}E_\emptyset^{(\ell)}
  \end{gather*}
  and, for all $i\in I_R$,
  \begin{gather*}
      \nu_{R,i}\propto \nu_i+\sum_{k\geq 0}\bar\theta^{-k-1}\P_{\nu_i}(X_{1}\not\in E_i,\ X_{k+1}\in\cdot)  \\
      \eta_{R,i}(x)>0\text{ for all }x\in D_2\text{ with }\mathbb P_x(T_{E_i}<\infty)>0,
  \end{gather*}
  and
  \[
      \eta_{R,i}(x)=0\text{ for all }x\in D_2\text{ with }\mathbb P_x(T_{E_i}<\infty)=0.
  \]

  In addition, similarly to Step~1, one checks that the first and second part of~\eqref{eq:A2} hold true. In particular, for all
  $x\in D_1$,
  \[
    \mathbb{E}_x(W_\emptyset(Y^{(\emptyset)}_k)\mathbbm{1}_{Y^{(\emptyset)}_k\in E^{(n)}_\emptyset})\leq \mathbb{E}_x(W_\emptyset(Y^{(\emptyset)}_k)\mathbbm{1}_{Y^{(\emptyset)}_k\in
      E_\emptyset})\leq C\gamma^k W_\emptyset(x).
  \]
    We deduce that Assumption~(A2) holds true and we can thus apply Theorem~\ref{thm:A2}. Since, for all $x\in D_1$ there exists $i\in F_{n}$ such that  $\mathbb P_x(T_{E_i}<+\infty)>0$, we deduce that, for all $x\in D_1$, 
\[\max_{k\geq 0} j_R(\delta_x P_kQ)=n=\text{index}(x).\]
This and Theorem~\ref{thm:A2} proves that Property (P) is satisfied on the set $S'_n$.

\medskip\textbf{Step 3. Proof that (P)  is satisfied on a set of the form~\eqref{eq:set-1} assuming it is satisfied on a set of the
  form~\eqref{eq:set-2}.} Assume that Property (P) is satisfied on the set $S'_n$ for some $n\in\{1,\ldots,\bar\ell-1\}$. Let us
prove that is it satisfied on the set $S_{n+1}$. In this case, we aim to apply Theorem~\ref{thm:A3} with
  \begin{align*}
    D_1=\bigcup_{i\in F_{n+1}} E_i\text{ and }D_2=\bigcup_{\ell=0}^{n}\bigcup_{i\in F_\ell} E_i \cup \bigcup_{\ell=-1}^{n} E_\emptyset^{(\ell)}.
\end{align*}
Using our induction assumption, we deduce that Assumption~(A) holds true for the process restricted to $D_2$, with $j_{0,R}=n$ and $\theta_{0,R}=\bar\theta$. As in Step~1, it is also clear that Assumption~(A) holds true for the process restricted to $D_1$, with $j_{0,P}=0$ and $\theta_{0,P}=\bar\theta$. As in Step~1, we also observe that~\eqref{eq:hyp-A3} holds true with $W_R=W\11_{D_2}$ and $W_P=C_W W \11_{D_1}$.

In order to check that Assumption~(A3) holds true, it remains to prove~\eqref{eq:definelstar} and~\eqref{eq:definelstar-bis}, with
$\ell^*=n$. Since $j_{0,R}=n=\max_{x\in D_2} j_R(x)$, the  equality~\eqref{eq:definelstar} is immediate.
For~\eqref{eq:definelstar-bis}, we observe that, for all $x\in D_1$, $\text{index}(x)=n+1$ and hence that the process starting from
$x$ can reach a point of index $n$ (otherwise, its index would be smaller or equal to $n$ by definition of the sets $F_i$), and hence that $\mathbb P_x(\text{index}(X_{T_{D_2}})=n\text{ and }T_{D_2}<\infty)>0$. Since $\text{index}(y)=n$ implies that $j_R(y)=n$ and $\eta_R(y)>0$ (by induction assumption), we deduce that~\eqref{eq:definelstar-bis} also holds true.

We deduce that Assumption~(A3) holds true and we can thus apply Theorem~\ref{thm:A3}, which concludes the proof.
\end{proof}

\section{Discrete state spaces}
\label{sec:discretelyap}

Let $X=(X_n,n\in\ZZ_+)$ be a Markov chain on a discrete state space $D\cup\{\partial\}$, with $\partial$ absorbing.  
It is well known, when $X$ is aperiodic and irreducible, i.e.\ when $\mathbb P_x(\exists n\geq 0,\ X_n=y)>0$ for all $x,y\in D$,
that existence of a quasi-stationary distribution is implied by the existence of a Lyapunov type function (see for instance~\cite{FerrariKestenEtAl1996,ChampagnatVillemonais2017b}, see also~\cite{DoornPollett2013} for a general account on quasi-stationary distributions for discrete state space models). We show in this section that the irreducibility assumption can actually be removed entirely.

In the following result, we say that $X$ is aperiodic if all states in $D$ are aperiodic (with the usual convention that a state  $x\in D$ such that  $\mathbb P_x(\exists n\geq 0,\ X_n=x)=0$ is said aperiodic).

\begin{thm}
    \label{thm:reduciblediscrete}
Assume that $X$ is aperiodic, that there exists $x_0\in D$ such that $\mathbb P_{x_0}(\exists n\geq 0,\ X_n=x_0)>0$ and that there
exists a function $V:D\to[1,+\infty)$ such that $\{x\in D,\ V(x)\leq C\}$ is finite for all constants $C>0$,
$\mathbb{E}_x (V(X_1)\11_{1<\tau_\d})<+\infty$ for all $x\in D$ and
  \begin{align}
\label{as:lyap}
 \frac{\E_x(V(X_1)\11_{1<\tau_\d})}{V(x)}\xrightarrow[V(x)\to+\infty]{} 0.
\end{align}
Then Assumption~(A) holds true with $W_S=V$ and, in particular, $X$ admits a quasi-stationary distribution. In addition, $\P_{x_0}(X_n\in\cdot\mid n<\tau_\d)$ converges in $\mathcal M(V)$ when $n\to+\infty$ toward a quasi-stationary distribution of $X$. 
\end{thm}

\begin{rem}
    Despite its generality, the assumption that there exists $x_0\in D$ such that $\mathbb P_{x_0}(\exists n\geq 0,\ X_n=x_0)>0$ is actually not necessary for the existence of a quasi-stationary distribution. Consider for instance the process with $D=\{1,2,\ldots\}$ and $\partial =0$, with almost sure transition
    from $i$ to $i-1$ for all $i\geq 1$. Then, choosing $\nu(i)=\frac{\theta}{1-\theta}\theta^i$ for all $i\geq 1$ and any $\theta\in(0,1)$, we have
    \begin{align*}
    \mathbb P_\nu(X_1=i)=\nu(i+1)=\frac{\theta}{1-\theta}\theta^{i+1}=\theta\,\nu(i),
    \end{align*}
    so that $\nu$ is a quasi-stationary distribution.
    \erem
\end{rem}

\begin{rem}
    In~\eqref{as:lyap}, we assumed for simplicity that $ \frac{\E_x(V(X_1)\11_{1<\tau_\d})}{V(x)}\xrightarrow[V(x)\to+\infty]{} 0$.
    However, a straightforward adaptation of the proof leads to a finer result: denoting by $C_i$, $i\in I$ with
    $I=\mathbb{N}:=\{1,2,\ldots\}$ or $I=\{1,\ldots,n\}$ for some $n\geq 1$, the collection of communication classes of the process, and by $\theta_i$ the exponential convergence parameter associated to each $C_i$, it is sufficient to assume that
    \begin{align*}
    \limsup_{V(x)\to+\infty} \frac{\E_x(V(X_1)\11_{1<\tau_\d})}{V(x)}<\sup_{i\in I}\theta_i.
    \end{align*}
    {\color{black} Another natural and straightforward adaptation of the result is to replace $V$ by any function $V':D\to[1,+\infty)$ without assuming that  $\{x\in D,\ V'(x)\leq C\}$ is finite for all $C\geq 0$, but such that, for a non-decreasing sequence of finite sets $(K_n)_{n\geq 0}$ such that $D=\cup_n K_n$, we have
    \begin{align*}
    \limsup_{n\to+\infty} \inf_{x\notin K_n}\frac{\E_x(V'(X_1)\11_{1<\tau_\d})}{V'(x)}<\sup_{i\in\mathbb N}\theta_i. 
    \end{align*}}
    \erem
\end{rem}

\begin{rem}
    The aperiodicity assumption is actually not needed for all $x\in D$: one easily checks that it is only required over
    communication classes whose exponential convergence parameter is maximal. More generally, adaptation of these results to periodic
    processes is common procedure {\color{black}(see e.g.~\cite{ChampagnatVillemonais2022})}, and we leave its details to the interested reader.
    \erem
\end{rem}

\begin{proof}[Proof of Theorem~\ref{thm:reduciblediscrete}]
    For all $x\in D$, let $C_x$ be the communication class of $x$, and let $(x_i)_{i\in I}$, where $I$ is either $\mathbb{N}$ or
    $\{1,\ldots,n\}$ for some $n$, be such that $D$ is the disjoint union of the sets $C_{x_i}$, $i\in I$. We take (without loss of generality) $x_1=x_0$ and write $C_i$ instead of $C_{x_i}$.
    
    Let $i\in I$ be such that $\mathbb P_{x_i}(\exists n\geq 0,\ X_n=x_i)>0$. By assumption, this is the case for $i=1$. Then the
    process $X$ restricted to $C_{i}$ is irreducible and satisfies Assumption~(E) in~\cite{ChampagnatVillemonais2016b} (this is a
    direct adaptation to the discrete time setting of the proof of Theorem~5.1 in the last reference). {\color{black}
      By~\cite[Corollary 2.7]{ChampagnatVillemonais2016b}, this implies that} the process $X$ restricted to $C_{i}$ satisfies
    Assumption~(A) with $j_S\equiv 0$, $\eta_S$ positive and $W_S=\restriction{V}{C_i}$. We denote by $\theta_i$ the associated
    exponential convergence parameter. In particular, it follows from~(A) that there exists a constant $A_i$ such that, for all
    $x\in C_i$ and all $n\geq 0$,
    \begin{equation}
      \label{eq:last-1}
      \mathbb{E}_x(V(X_n)\11_{n<T_{\{\d\}\cup D\setminus C_i}})\leq A_i\theta_i^n V(x).
    \end{equation}
    
    Let $i\in I$ be such that $\mathbb P_{x_i}(\exists n\geq 0,\ X_n=x_i)=0$. Then $C_{i}=\{x_i\}$.
    
    Now, define
    \[
    J:=\left\{i\in I,\  \frac{\E_x(V(X_1)\11_{1<\tau_\d})}{V(x)}<\theta_1\ \forall x\in C_i\right\}.
    \]
    By assumption, there exists only finitely many points $x\in D$ such that $\frac{\E_x(V(X_1)\11_{1<\tau_\d})}{V(x)}\geq \theta_1/2$, and hence there exists $\rho<\theta_1$ such that
    \[
    J=\left\{i\in I,\  \frac{\E_x(V(X_1)\11_{1<\tau_\d})}{V(x)}\leq\rho\ \forall x\in C_i\right\}
    \]
    In particular, for all $x\in \cup_{j\in J} C_j$, 
    \begin{align*}
    \frac{\mathbb E_x( V(X_1)\11_{X_1\in \cup_{j\in J} C_j})}{V(x)}\leq \rho,
    \end{align*}
    so we deduce from Markov's property that, for all $n\geq 1$, {\color{black} using the notation $\tau_J:=T_{\{\d\}\cup D\setminus \cup_{j\in J} C_j}$,
    \begin{equation}
      \label{eq:last-2}
      \mathbb E_x\left( V(X_n)\11_{n<\tau_J}\right)\leq \rho^n V(x).
    \end{equation}}
    Note that, by assumption, $I\setminus J$ is finite. Recall that all $i\in I\setminus J$ is either such that
    $\mathbb P_{x_i}(\exists n\geq 0,\ X_n=x_i)=0$ and $C_i=\{x_i\}$, or such that $\mathbb P_{x_i}(\exists n\geq 0,\ X_n=x_i)>0$,
    which implies that the process restricted to $C_{i}$ satisfies Assumption~(A) as above with exponential convergence parameter
    $\theta_i$.
    We then define
    \[
      J':=\left\{i \in I\setminus J\text{ such that }\mathbb P_{x_i}\left(\exists n\geq 0,\ X_n=x_i\right)=0\right\}
    \]
    and, setting $\bar\theta =\max_{i\in I\setminus(J\cup J')}\theta_i$, 
    \[
      J'':=\left\{i\in I\setminus(J\cup J')\text{ such that }\theta_i<\bar\theta\right\}.
    \]
    Since $J''$ is finite, $\hat\theta:=\sup_{j\in J''} \theta_i<\bar\theta$.
    
    We now set
    \[ 
      E_\emptyset=\bigcup_{j\in J\cup J'\cup J''} C_j 
    \]
    and enumerate the $C_i$, $i\in I\setminus (J\cup J'\cup J'')$, as $E_1,\ldots,E_k$. We shall apply
    Theorem~\ref{thm:finite-classes} to the partition of $D$ into the disjoint sets $E_\emptyset,E_1,\ldots,E_k$. Note that Assumptions (B1)
    and (B2) are satisfied for all $E_i$, $1\leq i\leq k$, with $W_i=V_{\vert E_i}$. Note also that, because of~\eqref{as:lyap} and
    since $V\geq 1$, Assumption~(B4) is satisfied with $W=V$, i.e.\ for all $x\in D$,
    \begin{equation}
      \label{eq:last-3}
      \mathbb E_x (V(X_1)\11_{1<\tau_\d})\leq {A}V(x)
    \end{equation}
    for some constant $A$.

    {\color{black} Let us now check that Assumption~(B3) is satisfied with $W_\emptyset=V_{\vert E_\emptyset}$. We set
      $\gamma_\emptyset=\rho\vee \hat \theta<\bar\theta$. Fix $n\geq 0$. Given any path $(X_k,0\leq k\leq n)$ of $X$ in
      $E_\emptyset$, we introduce an auxiliary process $(J_k,0\leq k\leq n)$ defined as follows: we set $J_k=j\in J'\cup J''$ if
      $X_k\in C_j$, and otherwise, we set $J_k=\aleph$. This means that $J_k=\aleph$ whenever $X_k\in C_j$ for any $j\in J$. Given
      any path $\mathbf{j}=(j_0,\ldots,j_n)\in (\{\aleph\}\cup J'\cup J'')^{n+1}$ of $(J_k,0\leq k\leq n)$, we denote by
      $n_\text{trans}(\mathbf{j})$ the number of transitions in the sequence $\mathbf{j}$, that is the number of
      $k\in\{0,\ldots,n-1\}$ such that $j_k\neq j_{k+1}$ and by $n''(\mathbf{j})$ the number of visits of $J''$ in $\mathbf{J}$, that
      is the number of pairs $(k,\ell)\in\{0,\ldots,n\}$ such that $k<\ell$, $j_k=j_{k+1}=\ldots=j_\ell\in J''$, $j_{k-1}\neq j_k$ or
      $k=0$ and $j_{\ell+1}\neq j_\ell$ or $\ell=n$.

      We shall prove by induction on $n_\text{trans}(\mathbf{j})$ that for all
      $\mathbf{j}\in \bigcup_{n\geq 0}(\{\aleph\}\cup J'\cup J'')^{n+1}$,
      \begin{equation}
        \label{eq:induction-last}
        \mathbb{E}_x\left(V(X_n)\mathbbm{1}_{(J_k,0\leq k\leq n)=\mathbf{j}}\right)\leq
        A^{n_\text{trans}(\mathbf{j})}\left(\max_{j\in J''}A_j \right)^{n''(\mathbf{J})}\gamma_\emptyset^{n-n_\text{trans}(\mathbf{j})}.
      \end{equation}
      First, if $\mathbf{j}\in(\{\aleph\}\cup J'\cup J'')^{n+1}$ is such that $n_\text{trans}(\mathbf{j})=0$, this means that
      $J_0=J_1=\ldots=J_n$. If $J_0\in J'$, this means that $n=0$, so~\eqref{eq:induction-last} is clear. If $J_0\in J''$, then
      $n''(\mathbf{j})=1$ and~\eqref{eq:induction-last} follows from~\eqref{eq:last-1}. If $J_0=\aleph$,~\eqref{eq:induction-last}
      follows from~\eqref{eq:last-2}.

      Assume now that we have proved~\eqref{eq:induction-last} for all $\mathbf{j}$ such that $n_\text{trans}(\mathbf{j})=k\geq 0$
      and let $\mathbf{j}\in(\{\aleph\}\cup J'\cup J'')^{n+1}$ be such that $n_\text{trans}(\mathbf{j})= k+1$. This means that
      $\mathbf{j}=(\mathbf{j}',j,j,\ldots,j)$ with $j\in \{\aleph\}\cup J'\cup J''$ repeated $\ell$ times for some $\ell\geq 1$ and
      $\mathbf{j}'\in(\{\aleph\}\cup J'\cup J'')^{n-\ell+1}$ is such that $n_\text{trans}(\mathbf{j}')=n$. If
      $j=\aleph$, it follows from Markov property that
    \begin{align*}
      \mathbb{E}_x\left(V(X_n)\mathbbm{1}_{(J_k,0\leq k\leq n)=\mathbf{j}}\right)
      & =\mathbb{E}_x\left[\mathbbm{1}_{(J_k,0\leq k\leq n-\ell+1)
        =(\mathbf{j}',j)}\mathbb{E}_{X_{n-\ell+1}}\left(V(X_{\ell-1})\mathbbm{1}_{J_p=j,\,0\leq
        p\leq \ell-1)=\mathbf{j}}\right)\right) \\ & \leq \gamma_\emptyset^\ell \mathbb{E}_x\left[\mathbbm{1}_{(J_k,0\leq k\leq n-\ell)
        =\mathbf{j}'}\mathbb{E}_{X_{n-\ell}} \left(V(X_1)\right)\right) \\ &\leq A \gamma_\emptyset^\ell
                                                                            \mathbb{E}_x\left[\mathbbm{1}_{(J_k,0\leq k\leq n-\ell)
        =\mathbf{j}'}V(X_{\ell'})\right),
    \end{align*}
    where we used~\eqref{eq:last-2} in second line and~\eqref{eq:last-3} in the last line. Observing that
    $n''(\mathbf{j})=n''(\mathbf{j}')$,~\eqref{eq:induction-last} for $\mathbf{j}$ follows from the induction assumption. We proceed similarly
    if $j\in J'$ using that $\ell=1$ and~\eqref{eq:last-3} and if $j\in J''$ using~\eqref{eq:last-1} and~\eqref{eq:last-3}.

    Let $n$ be fixed and let ${\cal P}_n$ be the set of $\mathbf{j}\in (\{\aleph\}\cup J'\cup J'')^{n+1}$ such that
    $\mathbb{P}_x((J_k,0\leq k\leq n)=\mathbf{j})>0$. For all $\mathbf{j}\in{\cal P}_n$, since the $C_j$ are communication classes
    for all $j\in J'\cup J''$, they are visited at most once by $\mathbf{j}$, that is there exists at most one $i\in\{0,\ldots,n-1\}$
    such that $j_i=j$ and $j_{i+1}\neq j$, and at most one $i'\in\{1,\ldots,n\}$ such that $j_{i'}\neq j$ and $j_{i'+1}=j$, and in
    addition if $j\in J'$, there exists at most one $i''\in\{0,\ldots,n\}$ such that $j_{i''}=j$. This means that, for all
    $\mathbf{j}\in{\cal P}_n$, $n''(\mathbf{j})\leq \# J''$ and $n_\text{trans}(\mathbf{j})\leq 2(\# J'+\# J'')$, and thus it follows
    from~\eqref{eq:induction-last} that
    \[
      \mathbb{E}_x\left(V(X_n)\mathbbm{1}_{(J_k,0\leq k\leq n)=\mathbf{j}}\right)\leq C'\gamma_\emptyset^{n}
    \]
    for a constant $C'$ independent of $n$.

    Now
    \[
      \# {\cal P}_n\leq (\# J'+\# J'')!\,(n+1)^{2\# J'+2\# J''+1}
    \]
    since, to construct a path $\mathbf{j}\in{\cal P}_n$, one must
    first choose an order of (possibly empty) visits of the classes $C_j$ for $j\in J'\cup J''$ and then one must choose the length of the (possibly
    empty) path in $\aleph$ before each of these visits, the length of this visit and the length of the (possibly empty) path in $\aleph$ after
    this visit, and they are all less than $n+1$. Therefore, given any $\gamma'_\emptyset\in(\gamma_\emptyset,\bar{\theta})$, there
    exists a constant $C''$ independent of $n$ such that
    \[
      \mathbb{E}_x\left(V(X_n)\mathbbm{1}_{n<T_{\{\d\}\cup E_1\cup\ldots\cup E_k}}\right)=\sum_{\mathbf{j}\in{\cal
            P}_n}\mathbb{E}_x\left(V(X_n)\mathbbm{1}_{(J_k,0\leq k\leq n)=\mathbf{j}}\right) \leq C'' (\gamma'_\emptyset)^n.
      \]
      Hence~(B3) is proved and we deduce from Theorem~\ref{thm:finite-classes} that $X$ satisfies Assumption~(A).
    }

    In order to prove the last statement of Theorem~\ref{thm:reduciblediscrete}, we apply the above proof to the process $X$ restricted to $D_{x_0}\cup\{\d\}$, where
    \[
    D_{x_0}=\{x\in D\text{ such that }\P_{x_0}(\exists n\geq 0,\ X_n=x)>0\}.
    \]
    We deduce from~\eqref{eq:thmfinite2} in Theorem~\ref{thm:finite-classes} that $\eta_S(x)>0$ in Assumption~(A) for this process,
    so that Proposition~\ref{prop:mainbis}(iii) entails the claim for $X$ restricted to $D_{x_0}\cup\{\d\}$. But the definition of
    $D_{x_0}$ clearly implies that $T_{\{\d\}\cup D\setminus D_{x_0}}=\tau_\d$ $\mathbb{P}_{x_0}$-a.s.,
which concludes the proof.
\end{proof}

\subsubsection*{Acknowledgements}

We thank two anonymous referees for their useful comments.

The work of N.C. is partially funded by the Chair ``Mod\'elisation Math\'ematique et
Biodiversit\'e'' of VEOLIA-Ecole Polytechnique-MNHN-F.X and by the European Union (ERC, SINGER, 101054787). Views and opinions expressed are however those of the author(s) only and do not necessarily reflect those of the European Union or the European Research Council. Neither the European Union nor the granting authority can be held responsible for them.

\def\cprime{$'$}

\end{document}

%% file: dessin.pdf_tex
\begingroup%
  \makeatletter%
  \providecommand\color[2][]{%
    \errmessage{(Inkscape) Color is used for the text in Inkscape, but the package 'color.sty' is not loaded}%
    \renewcommand\color[2][]{}%
  }%
  \providecommand\transparent[1]{%
    \errmessage{(Inkscape) Transparency is used (non-zero) for the text in Inkscape, but the package 'transparent.sty' is not loaded}%
    \renewcommand\transparent[1]{}%
  }%
  \providecommand\rotatebox[2]{#2}%
  \newcommand*\fsize{\dimexpr\f@size pt\relax}%
  \newcommand*\lineheight[1]{\fontsize{\fsize}{#1\fsize}\selectfont}%
  \ifx\svgwidth\undefined%
    \setlength{\unitlength}{425.83871014bp}%
    \ifx\svgscale\undefined%
      \relax%
    \else%
      \setlength{\unitlength}{\unitlength * \real{\svgscale}}%
    \fi%
  \else%
    \setlength{\unitlength}{\svgwidth}%
  \fi%
  \global\let\svgwidth\undefined%
  \global\let\svgscale\undefined%
  \makeatother%
  \begin{picture}(1,0.3989852)%
    \lineheight{1}%
    \setlength\tabcolsep{0pt}%
    \put(0,0){\includegraphics[width=\unitlength,page=1]{dessin.pdf}}%
    \put(0.18983045,0.19019588){\color[rgb]{0,0,0}\makebox(0,0)[lt]{\lineheight{1.25}\smash{\begin{tabular}[t]{l}$D_1$\end{tabular}}}}%
    \put(0.76359371,0.18930642){\color[rgb]{0,0,0}\makebox(0,0)[lt]{\lineheight{1.25}\smash{\begin{tabular}[t]{l}$D_2$\end{tabular}}}}%
    \put(0.48058861,0.00390773){\color[rgb]{0,0,0}\makebox(0,0)[lt]{\lineheight{1.25}\smash{\begin{tabular}[t]{l}$\partial$\end{tabular}}}}%
    \put(0,0){\includegraphics[width=\unitlength,page=2]{dessin.pdf}}%
    \put(0.47270917,0.2916053){\color[rgb]{0,0,0}\makebox(0,0)[lt]{\lineheight{1.25}\smash{\begin{tabular}[t]{l}$Q$\end{tabular}}}}%
    \put(0,0){\includegraphics[width=\unitlength,page=3]{dessin.pdf}}%
    \put(-0.00155942,0.3847119){\color[rgb]{0,0,0}\makebox(0,0)[lt]{\lineheight{1.25}\smash{\begin{tabular}[t]{l}$P$\end{tabular}}}}%
    \put(0.96460111,0.38219586){\color[rgb]{0,0,0}\makebox(0,0)[lt]{\lineheight{1.25}\smash{\begin{tabular}[t]{l}$R$\end{tabular}}}}%
    \put(0,0){\includegraphics[width=\unitlength,page=4]{dessin.pdf}}%
  \end{picture}%
\endgroup%

%% file: note-reducible_2025_10_16.bbl
\begin{thebibliography}{10}
    
    \bibitem{BansayeCloezEtAl2019}
    V.~{Bansaye}, B.~{Cloez}, P.~{Gabriel}, and A.~{Marguet}.
    \newblock {A non-conservative Harris' ergodic theorem}.
    \newblock {\em Journal of the London Mathematical Society},
    106(3):2459--2510, 2022.
    
    \bibitem{Barnes2005}
    B.~Barnes.
    \newblock Riesz points of upper triangular operator matrices.
    \newblock {\em Proceedings of the American Mathematical Society},
    133(5):1343--1347, 2005.
    
    \bibitem{BarraaBoumazgour2003}
    M.~Barraa and M.~Boumazgour.
    \newblock A note on the spectrum of an upper triangular operator matrix.
    \newblock {\em Proceedings of the American Mathematical Society},
    131(10):3083--3088, 2003.
    
    \bibitem{BenaimChampagnatEtAl2021}
    M.~Bena{\"\i}m, N.~Champagnat, W.~O{\c{c}}afrain, and D.~Villemonais.
    \newblock Degenerate processes killed at the boundary of a domain.
    \newblock to appear in {\em Annals of Probability}.

    \bibitem{BenaimCloez2015}
    M.~Benaim, and B.~Cloez.  
    \newblock A stochastic approximation approach to quasi-stationary distributions on finite spaces .
    \newblock {\em Electron. Commun. Probab.}, 20:1--13, 2015.  
      
    \bibitem{BenaimCloezEtAl2016}
    M.~Benaim, B.~Cloez, and F.~Panloup.  
    \newblock Stochastic approximation of quasi-stationary distributions on compact spaces and applications.  
    \newblock {\em The Annals of Applied Probability}, 28(4):2370--2416, 2018.  
    
    \bibitem{BenhidaZeroualiEtAl2005}
    C.~Benhida, E.~Zerouali, and H.~Zguitti.
    \newblock Spectra of upper triangular operator matrices.
    \newblock {\em Proceedings of the American Mathematical Society},
    133(10):3013--3020, 2005.

    \bibitem{CattiauxMeleard2010}
      P. Cattiaux, and S. M{\'e}l{\'e}ard.
    \newblock Competitive or weak cooperative stochastic Lotka--Volterra systems conditioned on non-extinction 
    \newblock {\em Journal of mathematical biology},
    60(6): 797--829, 2010.
    
    \bibitem{ChampagnatDiaconisEtAl2012}
    N.~Champagnat, P.~Diaconis, and L.~Miclo.
    \newblock On {D}irichlet eigenvectors for neutral two-dimensional {M}arkov
    chains.
    \newblock {\em Electron. J. Probab.}, 17:no. 63, 41, 2012.
    
    \bibitem{ChampagnatRoelly2008}
    N.~Champagnat and S.~R{\oe}lly.
    \newblock Limit theorems for conditioned multitype {D}awson-{W}atanabe
    processes and {F}eller diffusions.
    \newblock {\em Electron. J. Probab.}, 13:no. 25, 777--810, 2008.
    
    \bibitem{ChampagnatVillemonais2016b}
    N.~Champagnat and D.~Villemonais.
    \newblock Exponential convergence to quasi-stationary distribution and
    {Q}-process.
    \newblock {\em Probab. Theory Related Fields}, 164(1):243--283, 2016.
    
    \bibitem{ChampagnatVillemonais2017b}
N.~Champagnat and D.~Villemonais.  
\newblock General criteria for the study of quasi-stationarity.  
\newblock {\em Electronic Journal of Probability}, 28:1--84, 2023.  

    
    \bibitem{ChampagnatVillemonais2019}
    N.~{Champagnat} and D.~{Villemonais}.
    \newblock {Practical criteria for R-positive recurrence of unbounded
        semigroups}.
    \newblock {\em Electronic Communications in Probability}, 25(6):1--11, 2020.

  \bibitem{ChampagnatVillemonais2022}
   N.~{Champagnat} and D.~{Villemonais}.
    \newblock {Quasi-limiting estimates for periodic absorbed Markov chains}.
    \newblock {\em ArXiv preprint arXiv:2211.02706}, 2022.
    
    \bibitem{ColletMartinezEtAl2013}
    P.~Collet, S.~Mart\'inez, and J.~San~Mart\'in.
    \newblock {\em Quasi-stationary distributions}.
    \newblock Probability and its Applications (New York). Springer, Heidelberg,
    2013.
    \newblock Markov chains, diffusions and dynamical systems.
    
    \bibitem{darroch-seneta-65}
    J.~N. Darroch and E.~Seneta.
    \newblock On quasi-stationary distributions in absorbing discrete-time finite
    {M}arkov chains.
    \newblock {\em J. Appl. Probab.}, 2:88--100, 1965.

   \bibitem{delmas-et-al-2025}
    J.-F.\ Delmas, K.\ Lefki and P.-A.\ Zitt
    \newblock Atoms and associated spectral properties for positive operators on $L^p$ 
    \newblock {\em Pacific Journal of Mathematics}, 337(1):87--136, 2025.
    
    \bibitem{FerrariKestenEtAl1996}
    P.~A. Ferrari, H.~Kesten, and S.~Mart{\'{\i}}nez.
    \newblock {$R$}-positivity, quasi-stationary distributions and ratio limit
    theorems for a class of probabilistic automata.
    \newblock {\em Ann. Appl. Probab.}, 6(2):577--616, 1996.
    
    \bibitem{FerreRoussetStoltz2018}
G.~Ferré, M.~Rousset, and G.~Stoltz.  
\newblock More on the long time stability of Feynman–Kac semigroups.  
\newblock {\em Stochastics and Partial Differential Equations: Analysis and Computations}, 9(3):630--673, 2021.  
    
    \bibitem{GongQianEtAl1988}
    G.~L. Gong, M.~P. Qian, and Z.~X. Zhao.
    \newblock Killed diffusions and their conditioning.
    \newblock {\em Probab. Theory Related Fields}, 80(1):151--167, 1988.
    
    \bibitem{Gosselin2001}
    F.~Gosselin.
    \newblock Asymptotic behavior of absorbing {M}arkov chains conditional on
    nonabsorption for applications in conservation biology.
    \newblock {\em Ann. Appl. Probab.}, 11(1):261--284, 2001.
    
    \bibitem{GuillinNectouxEtAl2020}
A.~Guillin, B.~Nectoux, and L.~Wu.  
\newblock Quasi-stationary distribution for strongly Feller Markov processes by Lyapunov functions and applications to hypoelliptic Hamiltonian systems.  
\newblock {\em Journal of the European Mathematical Society}, 2024.  

    
    \bibitem{GuillinNectouxEtAl2021}
A.~Guillin, B.~Nectoux, and L.~Wu.  
\newblock Quasi-stationary distribution for Hamiltonian dynamics with singular potentials.  
\newblock {\em Probability Theory and Related Fields}, 185(3):921--959, 2023.  

    
    \bibitem{HinrichKolbEtAl2018}
G.~Hinrichs, M.~Kolb, and V.~Wachtel.  
\newblock Persistence of one-dimensional AR(1)-sequences.  
\newblock {\em Journal of Theoretical Probability}, 33(1):65--102, 2020.  

    
    \bibitem{LelievreRamilEtAl2021}
T.~Lelièvre, M.~Ramil, and J.~Reygner.  
\newblock Quasi-stationary distribution for the Langevin process in cylindrical domains, part I: existence, uniqueness and long-time convergence.  
\newblock {\em Stochastic Processes and their Applications}, 144:173--201, 2022.  

    \bibitem{Mandl1959}
    P.~Mandl.
    \newblock Sur le comportement asymptotique des probabilit\'{e}s dans les
    ensembles des \'{e}tats d'une cha\^{\i}ne de {M}arkov homog\`ene.
    \newblock {\em \v{C}asopis P\v{e}st. Mat.}, 84:140--149, 1959.
    
    \bibitem{MeleardVillemonais2012}
    S.~M{\'e}l{\'e}ard and D.~Villemonais.
    \newblock Quasi-stationary distributions and population processes.
    \newblock {\em Probab. Surv.}, 9:340--410, 2012.
    
    \bibitem{MeynTweedie2009}
    S.~Meyn and R.~L. Tweedie.
    \newblock {\em Markov chains and stochastic stability}.
    \newblock Cambridge University Press, Cambridge, second edition, 2009.
    \newblock With a prologue by Peter W. Glynn.
    
    \bibitem{NiemiNummelin1986}
    S.~Niemi and E.~Nummelin.
    \newblock On non-singular renewal kernels with an application to a semigroup of
    transition kernels.
    \newblock {\em Stochastic processes and their applications}, 22(2):177--202,
    1986.
    
    \bibitem{Ogura1975}
    Y.~Ogura.
    \newblock Asymptotic behavior of multitype {G}alton-{W}atson processes.
    \newblock {\em J. Math. Kyoto Univ.}, 15(2):251--302, 1975.
    
    \bibitem{DoornPollett2008}
    E.~A. van Doorn and P.~K. Pollett.
    \newblock Survival in a quasi-death process.
    \newblock {\em Linear Algebra and its Applications}, 429(4):776 -- 791, 2008.
    
    \bibitem{DoornPollett2009}
    E.~A. van Doorn and P.~K. Pollett.
    \newblock Quasi-stationary distributions for reducible absorbing {M}arkov
    chains in discrete time.
    \newblock {\em Markov Process. Related Fields}, 15(2):191--204, 2009.
    
    \bibitem{DoornPollett2013}
    E.~A. van Doorn and P.~K. Pollett.
    \newblock Quasi-stationary distributions for discrete-state models.
    \newblock {\em European J. Oper. Res.}, 230(1):1--14, 2013.
    
    \bibitem{Zhang2013}
    H.~Zhang.
    \newblock Spectra of 2$\times$ 2 upper-triangular operator matrices.
    \newblock {\em Applied Mathematics}, 4(11A):22, 2013.
    
\end{thebibliography}
